\documentclass[12pt,reqno]{amsart}
\usepackage{fullpage,  amssymb, amsthm, amsmath, amsfonts, times}

\usepackage{enumitem}
\usepackage{mathrsfs}
\usepackage{mathtools}
\usepackage{centernot}
\usepackage{xfrac}
\usepackage{eufrak}
\usepackage{stmaryrd}
\usepackage{bm}
\usepackage[all,cmtip]{xy}
\usepackage[usenames,dvipsnames]{xcolor}
\usepackage[margin=1.0 in]{geometry}
\usepackage{float}
\usepackage{hyperref}
\hypersetup{colorlinks=true, citecolor=blue, linkcolor=blue, urlcolor=blue, pdfstartview=FitH, pdfauthor=Curran, pdftitle=}

\usepackage{tikz-cd}
\usetikzlibrary{positioning,arrows}
\usetikzlibrary{calc}

\usepackage{todonotes}

\usepackage[utf8]{inputenc}

\begin{document}

\parskip0pt
\parindent10pt

\newenvironment{answer}{\color{Blue}}{\color{Black}}
\newenvironment{exercise}
{\color{Blue}\begin{exr}}{\end{exr}\color{Black}}

\theoremstyle{plain} 
\newtheorem{theorem}{Theorem}[section]
\newtheorem*{theorem*}{Theorem}
\newtheorem{prop}[theorem]{Proposition}
\newtheorem{porism}[theorem]{Porism}
\newtheorem{lemma}[theorem]{Lemma}
\newtheorem{cor}[theorem]{Corollary}
\newtheorem{conj}[theorem]{Conjecture}
\newtheorem{funfact}[theorem]{Fun Fact}
\newtheorem*{claim}{Claim}
\newtheorem{question}{Question}
\newtheorem*{conv}{Convention}

\theoremstyle{remark}
\newtheorem{exr}{Exercise}
\newtheorem*{rmk}{Remark}

\theoremstyle{definition}
\newtheorem{defn}{Definition}
\newtheorem{example}{Example}

\renewcommand{\mod}[1]{{\ifmmode\text{\rm\ (mod~$#1$)}\else\discretionary{}{}{\hbox{ }}\rm(mod~$#1$)\fi}}

\newcommand{\ns}{\mathrel{\unlhd}}
\newcommand{\tr}{\text{tr}}
\newcommand{\wt}[1]{\widetilde{#1}}
\newcommand{\wh}[1]{\widehat{#1}}
\newcommand{\cbrt}[1]{\sqrt[3]{#1}}
\newcommand{\floor}[1]{\left\lfloor#1\right\rfloor}
\newcommand{\abs}[1]{\left|#1\right|}
\newcommand{\ds}{\displaystyle}
\newcommand{\nn}{\nonumber}
\newcommand{\re}{\text{Re}}
\renewcommand{\ker}{\textup{ker }}
\renewcommand{\char}{\textup{char }}
\renewcommand{\Im}{\textup{Im }}
\renewcommand{\Re}{\textup{Re }}
\newcommand{\area}{\textup{area }}
\newcommand{\isom}
    {\ds \mathop{\longrightarrow}^{\sim}}
\renewcommand{\ni}{\noindent}
\renewcommand{\bar}{\overline}
\newcommand{\morph}[1]
    {\ds \mathop{\longrightarrow}^{#1}}

\newcommand{\Gal}{\textup{Gal}}
\newcommand{\Aut}{\textup{Aut}}
\newcommand{\Crypt}{\textup{Crypt}}
\newcommand{\disc}{\textup{disc}}
\newcommand{\sgn}{\textup{sgn}}
\newcommand{\del}{\partial}

\newcommand{\mattwo}[4]{
\begin{pmatrix} #1 & #2 \\ #3 & #4 \end{pmatrix}
}

\newcommand{\vtwo}[2]{
\begin{pmatrix} #1 \\ #2 \end{pmatrix}
}
\newcommand{\vthree}[3]{
\begin{pmatrix} #1 \\ #2 \\ #3 \end{pmatrix}
}
\newcommand{\vcol}[3]{
\begin{pmatrix} #1 \\ #2 \\ \vdots \\ #3 \end{pmatrix}
}

\newcommand*\wb[3]{%
  {\fontsize{#1}{#2}\usefont{U}{webo}{xl}{n}#3}}

\newcommand\myasterismi{%
  \par\bigskip\noindent\hfill
  \wb{10}{12}{I}\hfill\null\par\bigskip
}
\newcommand\myasterismii{%
  \par\bigskip\noindent\hfill
  \wb{15}{18}{UV}\hfill\null\par\medskip
}
\newcommand\myasterismiii{%
  \par\bigskip\noindent\hfill
  \wb{15}{18}{z}\hfill\null\par\bigskip
}

\newcommand{\one}{{\rm 1\hspace*{-0.4ex} \rule{0.1ex}{1.52ex}\hspace*{0.2ex}}}

\renewcommand{\v}{\vec{v}}
\newcommand{\w}{\vec{w}}
\newcommand{\e}{\vec{e}}
\newcommand{\m}{\vec{m}}
\renewcommand{\u}{\vec{u}}
\newcommand{\vecx}{\vec{e}_1}
\newcommand{\vecy}{\vec{e}_2}
\newcommand{\vo}{\vec{v}_1}
\newcommand{\vt}{\vec{v}_2}

\renewcommand{\o}{\omega}
\renewcommand{\a}{\alpha}
\renewcommand{\b}{\beta}
\newcommand{\g}{\gamma}
\newcommand{\sig}{\sigma}
\renewcommand{\d}{\delta}
\renewcommand{\t}{\theta}
\renewcommand{\k}{\kappa}
\newcommand{\ve}{\varepsilon}
\newcommand{\op}{\text{op}}

\newcommand{\Z}{\mathbb Z}
\newcommand{\ZN}{\Z_N}
\newcommand{\Q}{\mathbb Q}
\newcommand{\N}{\mathbb N}
\newcommand{\R}{\mathbb R}
\newcommand{\C}{\mathbb C}
\newcommand{\F}{\mathbb F}
\newcommand{\T}{\mathbb T}
\renewcommand{\H}{\mathbb H}
\newcommand{\B}{\mathcal B}
\newcommand{\p}{\mathcal P}
\renewcommand{\P}{\mathbb P}
\renewcommand{\r}{\mathcal R}
\renewcommand{\c}{\mathcal C}
\newcommand{\h}{\mathcal H}
\newcommand{\f}{\mathcal F}
\newcommand{\s}{\mathcal S}
\renewcommand{\L}{\mathcal L}
\newcommand{\lam}{\lambda}
\newcommand{\E}{\mathcal E}
\newcommand{\Ex}{\mathbb E}
\newcommand{\D}{\mathbb D}
\newcommand{\oh}{\mathcal O}
\newcommand{\n}{\mathcal N}
\newcommand{\I}{\mathcal I}

\newcommand{\diam}{\text{ diam}}
\newcommand{\vol}{\text{vol}}
\newcommand{\Int}{\text{Int}}

\newcommand{\Span}{\text{span}}
\newcommand{\AP}{\text{AP}}

\newcommand{\MoM}{\text{MoM}}

\newcommand{\0}{{\vec 0}}

\newcommand{\ignore}[1]{}

\newcommand{\poly}[1]{\textup{Poly}_{#1}}

\newcommand*\circled[1]{\tikz[baseline=(char.base)]{
            \node[shape=circle,draw,inner sep=2pt] (char) {#1};}}

\newcommand*\squared[1]{\tikz[baseline=(char.base)]{
            \node[shape=rectangle,draw,inner sep=2pt] (char) {#1};}}

\title{Freezing transition and moments of moments of the Riemann zeta function} 

\author{Michael J. Curran}
\email{Michael.Curran@maths.ox.ac.uk}
\address{Mathematical Institute, University of Oxford, Oxford, OX2 6GG, United Kingdom.}

\maketitle

\begin{abstract}
Moments of moments of the Riemann zeta function, defined by
\[
\text{MoM}_T(k,\beta) := \frac{1}{T}\int_T^{2T} \Bigg(\int\displaylimits_{
|h|\leq (\log T)^\theta}|\zeta(\tfrac{1}{2} + i t + ih)|^{2\beta}~dh\Bigg)^k~dt
\]
where $k,\beta \geq 0$ and $\theta > -1$, were introduced by Fyodorov and Keating \cite{FK} when comparing extreme values of zeta in short intervals to those of characteristic polynomials of random unitary matrices.
We study the $k = 2$ case as $T \rightarrow \infty$ and obtain sharp upper bounds for $\MoM_T(2,\b)$ for all real $0\leq \b \leq 1$ as well as lower bounds of the conjectured order for all $\b \geq 0$.
In particular, we show that the second moment of moments undergoes a freezing phase transition with critical exponent $\b = \tfrac{1}{\sqrt{2}}$.
\end{abstract}

\section{Introduction}

Moments of the Riemann zeta function and other $L$-functions are classical objects in analytic number theory utilized in the study large values and the value distribution of $L$-functions.
It has long been conjectured that for $\b \geq 0$ there is an asymptotic
\[
\frac{1}{T}\int_T^{2T} |\zeta(\tfrac{1}{2}+ i t)|^{2\b} \sim C_\b (\log T)^{\b^2}
\]
as $T \rightarrow \infty$ for certain constants $C_\b$.
The values of the constants $C_\b$ were conjectured recently in the work of Keating and Snaith \cite{KS} using analogies between the Riemann zeta function and characteristic polynomials of random matrices.
These asymptotics have only been established in the case of $\b = 1$ due to work of Hardy and Littlewood \cite{HL} and the case of $\b = 2$ in work of Ingham \cite{InghamMV}, however.
Upper bounds of the correct order are known for real $\b \leq 2$ due to work of Heap, Radziwiłł and Soundararajan \cite{HRS}, and are known for all $\b \geq 0$ assuming the Riemann hypothesis due to Harper \cite{Harper}.
Lower bounds are known unconditionally for all $\b \geq 0$ due to  Heap and Soundararajan \cite{HS}.

While relatively little is known unconditionally about the moments of zeta for $\b >  2$, in recent literature there has been interest in moments and the value distribution of zeta in short intervals on the critical line. 
Instead of studying zeta along an interval $[T,2T]$, one instead chooses a point $t\in [T,2T]$ uniformly at random and then tries to understand the typical behavior of the moments or maximum of zeta in the shorter interval $[t-1, t + 1]$, say.
The question of the maximum of zeta in short intervals is the subject of the Fyodorov-Hiary-Keating conjecture \cite{FHK, FK}, which predicts that if $t$ is chosen uniformly at random from $[T,2T]$ then as $T \rightarrow \infty$
\begin{equation}\label{eqn:FHKConj}
\P\left(\max_{|h-t|\leq 1}  |\zeta(\tfrac{1}{2} + i t + i h)| > e^y\frac{\log T }{(\log \log T)^{3/4}}  \right) \rightarrow 1 - F(y),
\end{equation}
where $F(y)$ is a cumulative distribution function with tail decay $1-F(y) \sim ye^{-2y}$ as $y\rightarrow \infty$.
The fraction $\frac{3}{4}$ and the tail decay are manifestations of correlations of nearby shifts of zeta. 
If these correlations were not present, probabilistic heuristics would predict a $\tfrac{1}{4}$ in place of $\tfrac{3}{4}$ and a tail decay of $e^{-2y}$ instead of $ye^{-2y}$.
An analogous conjecture is known to hold for the circular-$\b$ ensemble in the random matrix setting-- refer to the work Paquette and Zeitouni \cite{PZCBE, PZCUE}.

There has been a lot of recent progress toward the conjecture (\ref{eqn:FHKConj}).
Conditionally on the Riemann hypothesis, Najnudel showed  with probability $1-o_\ve(1)$ that   $\max_{|h-t|\leq 1}  |\zeta(\tfrac{1}{2} + i t + i h)|$ is bounded by $(\log T)^{1+ \ve}$, as well corresponding upper and lower bounds for the imaginary part of $\log \zeta$ \cite{Najnudel}.
Unconditionally Arguin, Belius, Bourgade, Radziwiłł, and Soundararajan \cite{FHKLeadingOrder} proved  that $\max_{|h-t|\leq 1}  |\zeta(\tfrac{1}{2} + i t + i h)|$  lies in the interval $[(\log T)^{1-\ve},(\log T)^{1+ \ve}]$ with probability $1-o_\ve(1)$.
Harper \cite{HarperPartition} then refined the upper bound showing that with probability $1-o(1)$ 
\[
\max_{|h-t|\leq 1}  |\zeta(\tfrac{1}{2} + i t + i h)| \leq \frac{\log T}{(\log \log T)^{3/4 + o(1)}}.
\]
Most recently Arguin, Bourgade, and Radziwiłł established the upper bound in the Fyodorov-Hiary-Keating conjecture \cite{FHKUpper}. 
Indeed, one consequence of \cite{FHKUpper} is  that for fixed $y$ the probability in (\ref{eqn:FHKConj}) is $O(y e^{-2y})$ as $T \rightarrow \infty$.

Moments of zeta in short intervals were also recently studied Fyodorov and Keating \cite{FK} followed by Arguin, Ouimet, and Radziwiłł \cite{AOR}.
The problem now is to understand the behavior of the short moments or the partition function 
\begin{equation}\label{eq:MSI}
\int_{|h| \leq (\log T)^\t} |\zeta(\tfrac{1}{2} + i t + i h)|^{2\b} dt
\end{equation}
where $\t > -1$ and $t$ is chosen uniformly at random from $[T,2T]$. 
The typical behavior of these short moments was investigated by Fyodorov and Keating \cite{FK} when $\t = 0$, and their behavior for all $\t > -1$ was determined by  Arguin, Ouimet, and Radziwiłł \cite{AOR}.
One particularly interesting conclusion of the analysis is that there is a so-called freezing transition as one varies the parameter $\b$.
More precisely, for fixed $\t > -1$ there is a critical value $\b_c(\t)$ such that when $\b < \b_c(\t)$, the moments (\ref{eq:MSI}) are governed by the typical values of $\zeta(\tfrac{1}{2} + i t)$,
yet when $\b > \b_c(\t)$ the moments (\ref{eq:MSI}) are dominated by just a few large values of $\zeta(\tfrac{1}{2} + i t)$.

In this paper, we will not be interested in the typical values of the moments (\ref{eq:MSI}), but instead larger values.
More precisely, we will investigate the $k = 2$ case of the so-called \emph{moments of moments} of zeta:
\[
\MoM_T(k,\b) := \frac{1}{T}\int_T^{2T} \Bigg(~\int\displaylimits_{
|h|\leq (\log T)^\theta}|\zeta(\tfrac{1}{2} + i t + ih)|^{2\b}~dh\Bigg)^k~dt
\]
where $\b \geq 0$ and $\t > -1$.
In other terms, we  study the variance of the moments (\ref{eq:MSI}).
The problem of moments of moments was first studied by Fyodorov, Hiary, and Keating in the random matrix setting \cite{FHK, FK}, who conjectured asymptotics for
\[
\MoM_{U(N)}(k,\b) := \int_{U(N)}\left(\frac{1}{2\pi} \int_{0}^{2\pi} |P_{N}(A,\t)|^{2\b} d\t \right)^k dA
\]
as $N\rightarrow \infty$ where $dA$ is the Haar measure on the unitary group $U(N)$ of size $N$ and $P_N(A,\t) := \det(I - Ae^{-i\t})$ is the characteristic polynomial of $A$.
There has been a lot of recent progress on these conjectures, and it has now been proven that (see \cite{AKUnitary, BKUnitary, CK, Fahs}) for $\b \geq 0$ and $k \in \N$ that
\[
\MoM_{U(N)}(k,\b) \sim
\begin{cases}
A(k,\b) N^{k\b^2}   & k < 1/\b^2\\
B(k,\b) N^{k^2\b^2 - k + 1}   & k > 1/\b^2 
\end{cases}
\]
for certain constants $A(k,\b), B(k,\b)$ as $N \rightarrow\infty$.
For $k \geq 2$, the transition at the critical exponent $\b=1 / \sqrt{k}$ was analyzed by Keating and Wong \cite{KW}, who showed
\[
\MoM_{U(N)}(1/\b^2,\b) \sim  D_\b N \log N
\]
for certain explicit constants $D_\b$ that also appear in the theory of Gaussian multiplicative chaos.

In the case $\t = 0$, the random matrix philosophy of Keating and Snaith \cite{KS} leads us to the prediction that if  $k \geq 2$ is an integer and $\b \geq 0$ then
\[
\MoM_{T}(k,\b) \sim
\begin{cases}
A'(k,\b) (\log T)^{k\b^2}   & k < 1/\b^2\\
B'(k,\b) (\log T) (\log\log T )  & k = 1/\b^2\\
C'(k,\b) (\log T)^{k^2\b^2 - k + 1}   & k > 1/\b^2
\end{cases}
\]
for certain constants $A'(k,\b), B'(k,\b), C'(k,\b) $ as $T \rightarrow\infty$.
These conjectures were also shown to hold for $\b,k \in \N$ by Bailey and Keating \cite{BKZetaMoM} assuming the standard conjectures on shifted moments of zeta due to Conrey, Farmer, Keating, Rubinstein, and Snaith \cite{CFKRS}.
These asymptotics do not however include the critical or subcritical regimes where $\b \leq 1/\sqrt{k}$.

The goal of this paper is to provide upper and lower bounds of the expected order of magnitude in the case $k = 2$ and non-integer $\b$. 
In particular we will show that there is indeed a phase transition at $\b = \tfrac{1}{\sqrt{2}}$, although this transition occurs in the sub leading order terms when $\t > 0$.
The starting point of the method is to rewrite the integral as
\begin{equation}\label{eqn:MoMTensor}
\MoM_T(2,\b) = \frac{1}{T}\iint\displaylimits_{|h_1|,|h_2|\leq (\log T)^\t}  \int_T^{2T} |\zeta(\tfrac{1}{2} + i t + ih_1)\zeta(\tfrac{1}{2} + i t + ih_2)|^{2\b} ~dt ~dh_1 ~dh_2.
\end{equation}
Therefore we can understand the $\MoM_T(2,\b)$ by understanding the correlation structure of zeta in short intervals and then integrating over all of the possible shifts $h_1, h_2$.
Morally, we will find that if $|h_1 - h_2| \ll 1/\log T$ then 
\[
\frac{1}{T}\int_T^{2T} |\zeta(\tfrac{1}{2} + i t + ih_1)\zeta(\tfrac{1}{2} + i t + ih_2)|^{2\b} ~dt  \approx 
(\log T)^{4 \b^2} ,
\]
while when $|h_1 - h_2| \gg 1/\log T$ that 
\[
\frac{1}{T}\int_T^{2T} |\zeta(\tfrac{1}{2} + i t + ih_1)\zeta(\tfrac{1}{2} + i t + ih_2)|^{2\b} ~dt  \approx 
(\log T)^{2 \b^2} |\zeta(1 + i (h_1 - h_2))|^{2\b^2}.
\]
So $\zeta(\tfrac{1}{2} + i t + ih_1)$ and $\zeta(\tfrac{1}{2} + i t + ih_2)$ are almost perfectly correlated when $h_1$ and $h_2$ are close, and they decorrelate for more distant shifts. 
Indeed one may see that the factor of $ |\zeta(1 + i (h_1 - h_2))|^{2\b^2}$ is bounded on average by computing the moments of zeta on the one line.
A similar correlation structure for $\log |\zeta(\tfrac{1}{2} + i t)|$ was proven by Bourgade \cite{Bourgade}, and this correlation structure is what gives rise to the phase transition at $\b = 1/\sqrt{2}$.
However as we look at larger shifts, we also see that a large value of $\zeta(1 + i h)$ on the one line will cause
$\zeta(\tfrac{1}{2} + i t)$ and $\zeta(\tfrac{1}{2} + i t + ih)$ to have an unusually large average correlation.

For $\b = 1$, an asymptotic formula for $\MoM_T(2,1)$ as $T \rightarrow \infty$ is obtained in Kovaleva's thesis, which also studies the fourth moment of zeta with shifts as large as $T^{6/5-\ve}$ \cite{Kovaleva}.
However, this asymptotic does not address the particularly interesting critical case  $\b = 1/\sqrt{2}$.
Assuming RH, Chandee \cite{Chandee} studied more general shifted moments
\[
M_{\bm{\a},\bm{\b}}(T) = \int_T^{2T}  \prod_{k = 1}^m |\zeta(\tfrac{1}{2} + i (t + \a_k))|^{2 \b_k} dt
\]
where $\bm{\a} = \bm{\a}(T) =  (\a_1, \ldots, \a_m)$ and $\bm{\b} = (\b_1 \ldots , \b_m)$ satisfy $|\a_k| \leq T/2$, $\b_k\geq 0$, and $|\a_j - \a_k| = O(1)$. Under these assumptions Chandee obtained upper bounds of the conjectured order up to a $(\log T)^{\ve}$ loss.  Due to this loss, we are unable to detect the $\log \log T$ factor at the phase transition; however the work of Chandee could still be used to give upper bounds of the conjectured order up to a factor of $(\log T)^\ve$ for $\MoM_T(k,\b)$ when $-1 < \t \leq 0$, $\b \geq 0$ and $k \in \N$.
In the special case where $\bm{\b} = (\b, \b)$, Ng, Shen, and Wong \cite{NSW} removed the $(\log T)^\ve$ loss using the work of Harper \cite{Harper}.
This work is sufficient to obtain upper bounds of the conjectured order for $\MoM_T(2,\b)$ when $-1 < \t \leq 0$ for all $\b \geq 0$.
Ng, Shen, and Wong also gave upper bounds in the larger regime where $|\a_1 - \a_2| \leq T^{3/5}$; however these bounds are not sharp when $|\a_1 - \a_2|$ is unbounded and result in a loss of order $\log \log T$ when bounding $\MoM_T(2,\b)$ with $\t > 0$.

To unconditionally obtain upper bounds for $\MoM_T(2,\b)$ when $\b$ is not an integer, we will use a principle pioneered in works of Heap, Radziwiłł, and Soundararajan \cite{HRS} and Radziwiłł and Soundararajan \cite{RS-EC}; to obtain lower bounds we will use a principle in the works of Heap and Soundararajan \cite{HS} and Radziwiłł and Soundararajan \cite{RS-LB}.
Overall, these works demonstrate that if one can asymptotically evaluate the twisted $2k^{\text{th}}$ moment of a family of $L$-functions for some $k > 0$, then one can obtain upper bounds of the correct order for the  $2\b^{\text{th}}$ moment in that family when $\b \leq k$ as well as sharp lower bounds for for all $\b \geq 0$. This method has been used to give conjecturally sharp upper and lower bounds for moments of several families of $L$-functions; see \cite{Curran, HRS, HS, RS-LB, RS-EC} for example.

The first task of this paper will be to establish the corresponding twisted second and fourth moment formulae with shifts of order $(\log T)^\t$.
The second moment can be handled using a slight modification of the techniques used in Young's proof of Levinson's theorem \cite{YoungSL}, and the fourth moment will follow from some minor modifications to the work of Hughes and Young on the twisted fourth moment \cite{HYTwisted4}.
With these formulae in hand, we may then follow the principle of \cite{HRS} and \cite{RS-EC} to establish sharp upper bounds for $\MoM_T(2,\b)$ for all $\b \leq 1$, and lower  bounds for all $\b \geq 0$.

\begin{theorem}\label{thm:mainUpperBounds}
Let $-1 < \theta \leq 0$. If $0 \leq \beta < 1/\sqrt{2}$ then
\[
\MoM_T(2,\b) \ll 
\frac{1}{1-2\b^2} (\log T)^{2\b^2(1-\t) + 2\t}.
\]
For $1/\sqrt{2} < \b \leq 1$
\[
\MoM_T(2,\b) \ll 
\frac{1}{2\b^2-1} (\log T)^{4\b^2 + \t - 1},
\]
and for $\b = 1/\sqrt{2}$ 
\[
\MoM_T(2,\tfrac{1}{\sqrt{2}})  \ll (1 + \t) (\log T)^{1 + \t} (\log\log T).
\]
Instead if $\t > 0$, then if $0\leq \b \leq \min(\sqrt{(1+\t)/2}, 1)$
\[
\MoM_T(2,\b) \ll 
(\log T)^{2\b^2 + 2\t} + \frac{1}{1-2\b^2}(\log T)^{2\b^2 + \t},
\]
where the sub leading term is replaced by $(1+\t)(\log T)^{1+\t}(\log\log T)$ if $\b = 1/\sqrt{2}$. 
In the case $\min((\sqrt{1+\t})/2,1)\leq \b \leq 1$ we have
\[
\MoM_T(2,\b) \ll 
\frac{1}{2\b^2-1}(\log T)^{4\b^2 -1 + \t}.
\]
All implied constants are absolute assuming $T$ is taken sufficiently large in terms of $\theta$.
\end{theorem}

\begin{rmk}
If one assumes RH and that $\t \leq 0$, the restriction $\b \leq 1$ may be removed using the work of Ng, Shen, and Wong \cite{NSW}. 
For $\theta > 0$ and $0\leq \b \leq \min(\sqrt{(1+\t)/2}, 1)$, the sub leading order term can be removed provided one that $T$ is sufficiently large in terms of $\b$ or if one allows the implicit constant to depend on $\b$.
This technicality is necessary because the phase transition occurs in  the sub leading order terms of the second moment of moments.
\end{rmk}
\noindent

\begin{theorem}\label{thm:mainLowerBound}
Let $-1< \theta \leq 0$. If $0 \leq \beta < 1/\sqrt{2}$ then
\[
\MoM_T(2,\b) \gg_{\b}
(\log T)^{2\b^2(1-\t) + 2\t}.
\]
For $\b \geq 1/\sqrt{2}$
\[
\MoM_T(2,\b) \gg_{\b}
(\log T)^{4\b^2 + \t - 1},
\]
and for $\b = 1/\sqrt{2}$ 
\[
\MoM_T(2,\tfrac{1}{\sqrt{2}})  \gg (1 + \t) (\log T)^{1 + \t} (\log\log T).
\]
Instead if $\t > 0$, then if $0\leq \b \leq \sqrt{(1+\t)/2}$
\[
\MoM_T(2,\b) \gg_\b 
(\log T)^{2\b^2 + 2\t}.
\]
When $\b \geq \sqrt{(1+\t)/2}$, we have
\[
\MoM_T(2,\b) \gg_\b
\frac{1}{2\b^2-1}(\log T)^{4\b^2 -1 + \t},
\]
where all implied constants  depend only on $\b$ provided $T$ is large enough in terms of $\t$.
\end{theorem}
\noindent
Theorems \ref{thm:mainUpperBounds} and \ref{thm:mainLowerBound}  are consequences of the following result, which may be of independent interest.

\begin{theorem}\label{thm:shiftedMoments}
If $0 \leq \b \leq 1$, $\t > -1$ and $|h_1|,|h_2|\leq (\log T)^\t$ then
\[
\frac{1}{T}\int_T^{2T} |\zeta(\tfrac{1}{2} + i t + ih_1)\zeta(\tfrac{1}{2} + i t + ih_2)|^{2\b} ~dt \ll (\log T)^{2\b^2} |\zeta(1 + i (h_1 - h_2) + 1/\log T)|^{2\b^2}.
\]
For all $\b \geq 0$, we have a lower bound 
\[
\frac{1}{T}\int_T^{2T} |\zeta(\tfrac{1}{2} + i t + ih_1)\zeta(\tfrac{1}{2} + i t + ih_2)|^{2\b} ~dt \gg_\b (\log T)^{2\b^2} |\zeta(1 + i (h_1 - h_2) + 1/\log T)|^{2\b^2}.
\]
\end{theorem}
\noindent
This improves the result of Ng, Shen, and Wong \cite{NSW} and provides a lower bound under the more restrictive assumptions that $\b \leq 1$ and that the shifts are smaller than $(\log T)^\t$. 
This bound is of the same strength as the result of \cite{NSW} when $|h_1 - h_2|$ is bounded, but is stronger when $|h_1 - h_2|$ tends to infinity.

We will first establish the necessary twisted moment formulae in section \ref{sec:twistedMoments}.
Subsequently we will prove Theorem \ref{thm:mainUpperBounds} in section \ref{sec:upperBounds} and then Theorem \ref{thm:mainLowerBound} in section \ref{sec:lowerBounds}.
We now introduce some notation we will use throughout the paper.
The notation is very similar to that of \cite{HRS}, with only a few cosmetic differences.
Given some positive integer $\ell$ and parameters $T_0 < T_1 < \ldots < T_\ell$, define for $1\leq j \leq \ell$
\[
\p_j(s) := \Re \sum_{T_{j-1} < p \leq T_j } p^{-s}, \qquad P_j := \sum_{T_{j-1} < p \leq T_j } p^{-1}.
\]
If $T_\ell$ is chosen to be some small power of $T$, then the sum of all the $\p_j(s)$ will be a good proxy for $\log |\zeta(s)|$ on average \cite{HRS, HS}, and each $P_j$ can be thought of as the variance of each $\p_j(s)$ on the half line.
Then given parameters $K_j \geq 0$ for $1\leq j \leq \ell$ let
\[
\n_j(s;\b) := \sum_{m = 0}^{K_j} \frac{1}{m!} (\b \p_j(s) )^m = \sum_{\substack{p\mid n \Rightarrow p\in (T_{j-1},T_j]\\ \Omega(n) \leq K_j}} \frac{\b^{\Omega(n)} g(n)}{n^s}
\]
where $g(n)$ is the multiplicative function such that $g(p^a) = 1/a!$ and $\Omega(n)$ is the number of prime divisors of $n$ counting multiplicity.
Notice that $\n_j(s;\b)$ is simply a truncation of the series expansion for $\exp(\b \p_j(s))$.
Since $\p_j(s)$ is a reasonable proxy for $\log |\zeta(s)|$, if we choose parameters $K_j$ to be large multiples of the variances $P_j$ then one might expect
\[
\n(s;\b) := \prod_{j\leq \ell} \n_j(s;\b)
\]
to behave like $\zeta(s)^\b$ on average on the half line.
The main difficulty here is that if the $\p_j(s)$ are unusually large for some bad set of $s$, then the series approximation for $\exp(\b \p_j(s))$ will be weaker.
Fortunately, such events are rare, and we can discard the contribution of such bad $s$ using the incremental structure of $\n(s;\b)$.

\section*{Acknowledgements}
\noindent
The author would like to
thank his supervisor Jon Keating for suggesting this problem and for helpful discussions and feedback.

\section{Twisted moments with mesoscopic shifts}\label{sec:twistedMoments}

\subsection{Twisted second moment}
An asymptotic for the twisted second moment of zeta with unbounded shifts is not necessary to prove Theorems \ref{thm:mainUpperBounds} or \ref{thm:mainLowerBound}. However, such an estimate will imply an asymptotic for $\MoM_T(2,1)$ and will serve as a good warm up to the twisted fourth moment estimate we will require.

Before stating our twisted moment estimate, we introduce some convenient notation.
Denote
\[
\lam(s) = \pi^{s-1/2} \frac{\Gamma((1-s)/2)}{\Gamma(s/2)}.
\]
Therefore the functional equation for $\zeta$ can be expressed as $\zeta(s) = \lam(1-s) \zeta(1-s)$, and $\lam$ satisfies its own functional equation $\lam(s) \lam(1-s) = 1$.
Define
\[
G(s) = e^{s^2} \left(1 - \left(\frac{2s}{\a_1+ \a_2} \right)^2\right), 
\]
\[
g_{\a_1,\a_2}(s,t) = \pi^{-s} \frac{\Gamma\left(\frac{\tfrac{1}{2} + \a_1 + s + i t}{2}\right)\Gamma\left(\frac{\tfrac{1}{2} + \a_2 + s - i t}{2}\right)}{\Gamma\left(\frac{\tfrac{1}{2} + \a_1 + i t}{2}\right)\Gamma\left(\frac{\tfrac{1}{2} + \a_2  - i t}{2}\right)},
\]
\[
X_{\a_1,\a_2,t} =  \lam(\tfrac{1}{2} + \a_1  + i t) \lam(\tfrac{1}{2} + \a_2 - i t).
\]
The precise definition of $G$ here is not too important–
the key properties are that $G(0) = 1$, $G$ has a zero at $\tfrac{\a_1+\a_2}{2}$, and $G$ decays rapidly in any vertical strip.
By equation 4.12.3 of \cite{Titchmarsh}, we can approximate
\begin{align*}
X_{\a_1,\a_2,t} =   &\left(1 + O\left(\frac{1}{t}\right)\right)\\
\times \exp\Bigg((\a_1 + it) \log\left(\frac{2\pi}{t + \Im \a_1}\right) &+ (\a_2-i t) \log\left(\frac{2\pi}{t - \Im \a_2}\right) + \a_1-\a_2 \Bigg). 
\end{align*}
Therefore if we define
\[
Y_{\a_1,\a_2,t} := \exp\left(\frac{\a_1 + it}{2} \log\left(\frac{2\pi}{t + \Im \a_1}\right) + i\frac{\a_2-i t}{2}\log\left(\frac{2\pi}{t - \Im \a_2}\right) + \frac{\a_1-\a_2}{2} \right) 
\]
we have the approximation $X_{\a_1,\a_2}= Y_{\a_1,\a_2}^2(1 + O(1/t))$.
We will use this quantity to symmetrize our formula for the twisted second moment.
Finally, we write
\[
V_{\a_1,\a_2}(x,t) = \frac{1}{2\pi i} \int_{(1)} \frac{G(s)}{s} g_{\a_1,\a_2}(s,t)x^{-s} ~ ds.
\]
By Proposition 5.4 of \cite{IK} it follows that for $A \geq 0$, $j \in \Z_{\geq 0}$ and $T$ sufficiently large in terms of $\theta$
\begin{equation}\label{eqn:VSecondDecay}
t^j \frac{\del^j}{\del t^j} V_{\a_1,\a_2}(x,t)  \ll_{A,j} (1 + |x|/t)^{-A}.
\end{equation}
We will make use of the following version of the approximate functional equation given by theorem 5.3 of \cite{IK}.

\begin{lemma}\label{lem:AFE2nd}
Let $t\in [T/4,4T]$, $\Re \a_j \ll 1/ \log T$, $|\Im \a_j| \leq (\log T)^\theta$ for $j = 1,2$, and $\theta > -1$. Then for all $A > 0$ 
\begin{align*}
    \zeta(\tfrac{1}{2} + \a_1 + i t) \zeta(\tfrac{1}{2} + \a_2 - i t) = \sum_{m,n} \frac{1}{m^{1/2 + \a_1} n^{1/2 + \a_2}} \left(\frac{m}{n}\right)^{- i t} V_{\a_1,\a_2}(mn,t)\\
    + X_{\a_1,\a_2,t} \sum_{m,n} \frac{1}{m^{1/2 - \a_2} n^{1/2 - \a_1}} \left(\frac{m}{n}\right)^{- i t} V_{-\a_2,-\a_1}(mn,t) + O_A((1 + t)^{-A}).
\end{align*}
\end{lemma}

\noindent
We can now prove the necessary twisted second moment estimate.

\begin{theorem}\label{thm:twisted2nd}
Let $\theta > -1$, $w$ be a smooth non-negative function supported on $[1/2,4]$, and   
\[
A_{\a_1,\a_2}(s) = \sum_{n\leq T^\eta} \frac{a_{\a_1,\a_2}(n)}{n^s},~ B_{\a_1,\a_2}(s) = \sum_{n\leq T^\eta} \frac{b_{\a_1,\a_2}(n)}{n^s} 
\] 
be Dirichlet polynomials with $\eta < 1/2$ and $a_{\a_1,\a_2}(n), b_{\a_1,\a_2}(n) \ll_\ve n^\ve$. 
Then uniformly in $\Re \a_j \ll 1/ \log T$ and $|\Im \a_j| \leq (\log T)^\theta$ for $j = 1,2$ 
\begin{align*}
    \int_{\R} \zeta(\tfrac{1}{2} + \a_1 + it)\zeta(&\tfrac{1}{2} + \a_2 - it) A_{\a_1,\a_2}(\tfrac{1}{2} + i t) \overline{B_{\a_1,\a_2}(\tfrac{1}{2} + i t) } w(t/T) ~ dt \\
    = \frac{1}{(2\pi i)^2}\oint\displaylimits_{|z_j - \a_j| \asymp 1/\log T} &\frac{\zeta(1 + z_1 + z_2) (z_1 - z_2)^2}{(z_1-\a_1)(z_1-\a_2)(z_2-\a_1)(z_2-\a_2)}F_{\a_1,\a_2}(z_1,z_2)\\
    \times &\int_\R   Y_{\a_1,\a_2,t} ~ Y_{-z_2,-z_1,t} ~ w(t/T) ~ dt ~ dz_1 ~ dz_2 + O(T^{1-\d})
\end{align*}
for some $\d>0$ where 
\[
F_{\a_1,\a_2}(z_1,z_2) = \sum_{n,m\leq T^\eta} \frac{a_{\a_1,\a_2}(n) \overline{b_{\a_1,\a_2}(m)} }{[n,m]} \frac{(n,m)^{z_1+z_2}}{n^{z_2}m^{z_1}}.
\]
\end{theorem} 

\begin{rmk}
While there is be plenty of flexibility in the choice of contours about the shifts $\a_j$, one must take some care to ensure the contours do not intersect.
\end{rmk}

\begin{proof}
By Lemma \ref{lem:AFE2nd}, up to an arbitrary power savings in $T$ the integral of interest equals
\begin{align*}
&\sum_{h,k\leq T^\eta} \frac{a_{\a_1,\a_2}(h)\overline{b_{\a_1,\a_2}(k)}}{\sqrt{hk}}\sum_{m,n} \int_\R  \left(\frac{hm}{kn}\right)^{-it} V_{\a_1,\a_2}(mn,t)\frac{1}{m^{1/2+\a_1}n^{1/2+\a_2}} w(t/T)~dt \\
+ 
\sum_{h,k\leq T^\eta} &\frac{a_{\a_1,\a_2}(h)\overline{b_{\a_1,\a_2}(k)}}{\sqrt{hk}}\sum_{m,n} \int_\R X_{\a_1,\a_2,t}  \left(\frac{hm}{kn}\right)^{-it} V_{-\a_2,-\a_1}(mn,t)\frac{1}{m^{1/2-\a_2}n^{1/2-\a_1}} w(t/T)~dt  
\end{align*}
Inserting the definition of $V_{\a_1,\a_2}$ we see the diagonal terms with $hm = kn$ contribute 
\begin{align*}
&\quad\sum_{h,k\leq T^\eta} \frac{a_{\a_1,\a_2}(h)\overline{b_{\a_1,\a_2}(k)}}{2\pi i \sqrt{hk}} \int_\R  \int_{(1)}\frac{G(s)}{s}g_{\a_1,\a_2}(s,t) \sum_{hm = kn}\frac{1}{m^{1/2+\a_1+s}n^{1/2+\a_2+s}} w(t/T)~ds ~dt \\
+ 
&\sum_{h,k\leq T^\eta} \frac{a_{\a_1,\a_2}(h)\overline{b_{\a_1,\a_2}(k)}}{2\pi i \sqrt{hk}} \int_\R X_{\a_1,\a_2,t} \int_{(1)} \frac{G(s)}{s}g_{-\a_2,-\a_1}(s,t) \sum_{hm = kn}  \frac{1}{m^{1/2-\a_2 + s}n^{1/2-\a_1 + s}} w(t/T)~ ds ~dt 
\end{align*}
The inner Dirichlet series can be evaluated by parameterizing $m = \ell k/(h,k)$, $n = \ell h/(h,k)$ for $\ell \geq 1$ which gives
\[
\sum_{hm = kn}  \frac{1}{m^{1/2+\a_1 + s}n^{1/2+\a_2 + s}} = \frac{(h,k)^{1+\a_1 + \a_2 + 2s}}{h^{1/2+\a_2 + s} k^{1/2+\a_1+s}}\zeta(1 + \a_1 + \a_2 + 2s)
\]
for $\Re s = 1$.
Shifting the contour to $\Re s = -1/4+\ve$, we pick up a pole at $s = 0$ while the zero of $G(s)$ cancels out the pole of $\zeta(1+\a_1+\a_2 + 2 s)$.
Using Stirling's approximation and that the shifts have small real part  to bound the integral on $\Re s = -1/4+\ve$, we may simplify our expression to
\begin{align}\label{eqn:twisted2ndSum}
\sum_{h,k\leq T^\eta} &\frac{a_{\a_1,\a_2}(h)\overline{b_{\a_1,\a_2}(k)}}{[h,k]} \int_\R  \Bigg(\frac{(h,k)^{\a_1 + \a_2}}{h^{\a_2}k^{\a_1}}\zeta(1+\a_1+\a_2)  \\
&+ X_{\a_1,\a_2,t} \frac{(h,k)^{-\a_1 - \a_2}}{h^{-\a_1}k^{-\a_2}}\zeta(1-\a_1-\a_2)    \Bigg)w(t/T) ~dt + O\left(T^{1/2+\ve}\sum_{h,k\leq T^\eta} (hk)^{-1/4+\ve} \right) \nonumber.
\end{align}
The error term is $O(T^{1-\d})$ provided $\eta < 1/2$.
To evaluate the main term, we first replace $X_{\a_1,\a_2,t}$ by $Y_{\a_1,\a_2,t}^2$ at a negligible cost.
Because $Y_{\a_1,\a_2,t}^{-1} = Y_{-\a_2,-\a_1,t}$, we can now symmetrize our expression as
\begin{align}
\sum_{h,k\leq T^\eta} &\frac{a_{\a_1,\a_2}(h)\overline{b_{\a_1,\a_2}(k)}}{[h,k]} \int_\R  Y_{\a_1,\a_2,t} \Bigg( Y_{-\a_2,-\a_1,t} \frac{(h,k)^{\a_1 + \a_2}}{h^{\a_2}k^{\a_1}}\zeta(1+\a_1+\a_2)  \\
&+ Y_{\a_1,\a_2,t} \frac{(h,k)^{-\a_1 - \a_2}}{h^{-\a_1}k^{-\a_2}}\zeta(1-\a_1-\a_2)    \Bigg)w(t/T) ~dt + O\left(T^{1/2+\ve}\sum_{h,k\leq T^\eta} (hk)^{-1/4+\ve} \right) \nonumber.
\end{align}
Now an appeal to lemma 2.5.1 of \cite{CFKRS} permits us to rewrite the main term as the desired contour integral. 
To conclude, we must show that the off-diagonal terms with $hm \neq kn$ give a negligible contribution provided $\eta < 1/2$.
This follows by using repeated integration by parts and (\ref{eqn:VSecondDecay}).

\end{proof}

\subsection{Twisted fourth moment}
We now cover the corresponding notation for the twisted fourth moment estimate.
Write $\sigma_{z_1,z_2}(n) = \sum_{a b =  n} a^{-z_1} b^{-z_2}$ for the generalized divisor sum functions, and for the sake of concreteness let 
\[
G(s) = e^{s^2} \left(1 - \left(\frac{2s}{\a_1+ \a_3} \right)^2\right)\left(1 - \left(\frac{2s}{\a_1+ \a_4} \right)^2\right)\left(1 - \left(\frac{2s}{\a_2+ \a_3} \right)^2\right)\left(1 - \left(\frac{2s}{\a_2+ \a_4} \right)^2\right).
\]
For simplicity of notation, denote $\bm{\a} = (\a_1,\a_2,\a_3,\a_4)$ and $\pi\bm{\a} = (\a_3,\a_4,\a_1,\a_2)$. Write
\[
g_{\bm{\a}}(s,t) = g_{\a_1,\a_3}(s,t) g_{\a_2,\a_4} (s,t), 
\]

\[
X_{\bm{\a},t} = 
X_{\a_1,\a_3,t} ~ X_{\a_2,\a_4,t} ,\quad Y_{\bm{\a},t} = 
Y_{\a_1,\a_3,t} ~ Y_{\a_2,\a_4,t} ,
\]
and
\[
V_{\bm{\a}}(x,t) = \frac{1}{2\pi i} \int_{(1)} \frac{G(s)}{s} g_{\bm{\a}}(s,t)x^{-s} ds.
\]
As before, $X_{\bm{\a},t} = Y_{\bm{\a},t}^2(1+O(1/t))$ and we have the decay estimate
\begin{equation}\label{eqn:VFourthDecay}
t^j \frac{\del^j}{\del t^j} V_{\bm{\a}}(x,t)  \ll_{A,j} (1 + |x|/t^2)^{-A}.
\end{equation}
The appropriate approximate functional equation from theorem 5.3 of \cite{IK} takes the following form.

\begin{lemma}\label{lem:AFE4th}
Let $t\in [T/4,4T]$, $\Re \a_j \ll 1/ \log T$, $|\Im \a_j| \leq (\log T)^\theta$ for all $j$, and $\theta > -1$. Then for all $A > 0$ 
\begin{align*}
    \zeta(\tfrac{1}{2} + \a_1 + i t) \zeta(\tfrac{1}{2} + \a_2 + i t)\zeta(\tfrac{1}{2} + \a_3 - &i t) \zeta(\tfrac{1}{2} + \a_4 - i t)\\
    = \sum_{m,n} \frac{\sigma_{\a_1,\a_2}(m)\sigma_{\a_3,\a_4}(n)}{\sqrt{mn}} \left(\frac{m}{n}\right)^{- i t} &V_{\bm{\a}}( mn,t)\\
    + X_{\bm{\a},t} \sum_{m,n} \frac{\sigma_{-\a_3,-\a_4}(m)\sigma_{-\a_1,-\a_2}(n)}{\sqrt{mn}} \left(\frac{m}{n}\right)^{- i t} V_{-\pi\bm{\a}}&( mn,t) + O_A((1 + t)^{-A}).
\end{align*}
\end{lemma}

\noindent
Finally, we set
\[
B_{z_1,z_2,z_3,z_4}(n) = \prod_{p^m \|  n} \left(\sum_{j\geq 0} \frac{\sigma_{z_1,z_2}(p^{j}) \sigma_{z_3,z_4} (p^{j+m})}{p^j} \right)
\left(\sum_{j\geq 0} \frac{\sigma_{z_1,z_2}(p^{j}) \sigma_{z_3,z_4} (p^j)}{p^j} \right)^{-1},
\]

\[
A(z_1,z_2,z_3,z_4) = \frac{\zeta(1+z_1 + z_3)\zeta(1+z_1 + z_4)\zeta(1+z_2 + z_3)\zeta(1+z_2 + z_4)}{\zeta(2+z_1 + z_2+ z_3 + z_4)},
\]
and let
\[
\Delta(z_1,z_2,z_3,z_4) = \prod_{1\leq j < k \leq 4} (z_k - z_j)
\]
denote the Vandermonde determinant.

\begin{theorem}\label{thm:twisted4th}
Let $\theta > -1$, $w$ be a smooth non-negative function supported on $[1/2,4]$, and  
\[
A_{\bm{\a}}(s) = \sum_{n\leq T^\eta} \frac{a_{\bm{\a}}(n)}{n^s} \]
be a Dirichlet polynomial with $\eta < 1/11$ and $a_{\bm{\a}}(n) \ll_\ve n^\ve$. 
Then uniformly for $\Re \a_j \ll 1/ \log T$ and $|\Im \a_j| \leq (\log T)^\theta$
\begin{align*}
\int_{\R} \zeta(\tfrac{1}{2} + \a_1 + i t) \zeta(\tfrac{1}{2} &+ \a_2 + i t)\zeta(\tfrac{1}{2} + \a_3 - i t) \zeta(\tfrac{1}{2} + \a_4 - i t) |A_{\bm{\a}}(\tfrac{1}{2} + i t)|^2 w(t/T) ~ dt\\ 
= \frac{1}{4(2\pi i)^4}&\oint\displaylimits_{|z_j - \a_j| \asymp 1/\log T} A(\bm{z}) \Delta(z_1,z_2,-z_3,-z_4)^2 F_{\bm{\a}}(\bm{z})\\
&\int_\R  Y_{\bm{\a},t}~ Y_{-\pi\bm{z},t} ~w(t/T) ~ dt ~\prod_{ j \leq 4} \prod_{k\leq 4} \frac{dz_j}{(z_j-\a_k)}+ O(T^{1-\d})
\end{align*}
for some $\d > 0$, where $\bm{z} = (z_1,z_2,z_3,z_4)$, $-\pi\bm{z} = (-z_3,-z_4,-z_1,-z_2)$ and 
\[
F_{\bm{\a}}(\bm{z}) = \sum_{n,m\leq T^\eta} \frac{a_{\bm{\a}}(n) \overline{a_{\bm{\a}}(m)}}{[n,m]} B_{\bm{z}}\left(\frac{n}{(n,m)}\right) B_{-\pi\bm{z}}\left(\frac{m}{(n,m)}\right).
\]
\end{theorem}

\noindent By using Lemma \ref{lem:AFE4th}, one can prove that theorem 1.1 of \cite{HYTwisted4} is also valid for shifts with $\Re \a_j \ll 1/ \log T$ and $|\Im \a_j| \leq (\log T)^\theta$ provided one slightly modifies the main term.
We shall now sketch the necessary modifications.

Proceeding in the same manner of the proof of Theorem \ref{thm:twisted2nd}, we find that the diagonal term $I_D^{(1)}(h,k)$ defined by (38) of \cite{HYTwisted4} 
is given by
\[
I_D^{(1)}(h,k) = \frac{1}{\sqrt{h k}} \int_\R Z_{\a_1,\a_2,\a_3,\a_4,h,k}(0) \phi(t/T) dt + O(T^{1/2+\ve} (h k)^{-1/4}),
\]
where $Z_{\bm{\a},h,k}(s)$ is defined by (15) in \cite{HYTwisted4}.
Similarly, the term $I_D^{(2)}(h,k)$ defined by (39) of \cite{HYTwisted4} is given by 
\[
I_D^{(2)}(h,k) = \frac{1}{\sqrt{h k}} \int_\R Y_{\a_1,\a_2,\a_3,\a_4,t}^2  Z_{-\a_3,-\a_4,-\a_1,-\a_2,h,k}(0) \phi(t/T) dt + O(T^{1/2+\ve} (h k)^{-1/4}).
\]
Note we have also replaced $X_{\bm{\a},t}$ with $Y_{\bm{\a},t}^2$ which incurs a negligible cost. 
All of the $J_{\bm{\a}}$ terms defined by (56) and (57) of \cite{HYTwisted4}  vanish since the zeros of $G$ cancel out the poles of the integrand.

We now handle the off-diagonal contribution. 
Since the real parts of the shifts are small, the error term arising from the delta method in \cite{HYTwisted4}  is unchanged, but we have a slightly different main term.
The key modification is that  in place of (94) of \cite{HYTwisted4} one needs a more precise version of Stirling's approximation
\begin{align*}
\frac{\Gamma(\tfrac{1}{2} -\a_1 - s \pm i t)}{\Gamma(\tfrac{1}{2} +\a_3 + s \pm i t)}  &= \exp\bigg[ (\a_1 + s + i t) \log \left(1+ \frac{\a_1+ \a_3 + 2s}{\tfrac{1}{2} - \a_1 - s - it}\right) + \a_1 +\a_3 +2s \\
-(\a_1 + \a_3 + 2s&) \log(t-\Im \a_3) \mp (\a_1 + \a_3 + 2s) \frac{\pi i}{2}  \bigg] \left(1 + O\left(\frac{(1 + |s|^2 )(\log T)^{2\t}}{t}\right)\right),
\end{align*} 
which holds uniformly in $|\a_1|,|\a_3|\leq (\log T)^\t < T/10$ and $t\in [T/4,4T]$.
Notice also that the hypotheses of lemma 2 of \cite{BettinUS} are also satisfied in this range, so we may write
\begin{align*}
&\quad X_{\a_1,\a_3,t} = \lam(\tfrac{1}{2}+\a_1 + i t)\lam(\tfrac{1}{2} + \a_3 - i t)  =   \left(1 + O\left(\frac{(\log T)^{2\t}}{t}\right)\right)\\
\times \exp&\bigg(-(\a_1 + \a_3) \log \left(\frac{t- \Im \a_3}{2\pi} \right)+ (\a_1 + i t) 
\log\left(1+ \frac{\a_1+\a_3}{\tfrac{1}{2}-\a_1-it}\right) + \a_1 + \a_3 \bigg).
\end{align*}
One may then continue following the argument of section 6 of \cite{HYTwisted4} with $X_{\a_1,\a_3,t}$ in place of $(t/2\pi)^{-\a_1-\a_3}$. 
The expression for $P_{\bm{\a}}^{(0)}$ in proposition 6.8 of \cite{HYTwisted4} will now be replaced by
\[
P_{\a_1,\a_2,\a_3,\a_4}^{(0)} = \frac{1}{\sqrt{h k}} \int_\R 
Y_{\a_1,\a_3,t}^2 ~ Z_{-\a_3,\a_2,-\a_1,\a_4,h,k}(0) \phi(t/T) dt,
\]
where again we have replaced $X_{\a_1,\a_3,t}$ with $Y_{\a_1,\a_3,t}^2$ at a negligible cost.

Upon adding up all the diagonal and off-diagonal terms, we now conclude the proof by symmetrizing the integrand as we did for Theorem \ref{thm:twisted2nd}.
Simply note that we can divide out a factor of $Y_{\bm{\a}}$ from the integrand by using that $Y_{\g,\d,t}^{-1} = Y_{-\d,-\g,t}$ and 
\[
Y_{\a_1,\a_2,\a_3,\a_4}^{-1} = Y_{-\a_3,-\a_1,t}~Y_{-\a_4,-\a_2,t} = Y_{-\a_4,-\a_1,t}~Y_{-\a_3,-\a_2,t}.
\]
This puts the integrand in a form where lemma 2.5.1 of \cite{CFKRS} is applicable, and we may conclude by summing over $h$ and $k$.

\section{Upper Bounds}\label{sec:upperBounds}

To prove Theorem \ref{thm:mainUpperBounds}, we must estimate how zeta is correlated on average along intervals of size $(\log T)^\t$.
When the difference $|h_1 - h_2|$ is of size $\ll 1 / \log T$, one could simply appeal to a $4\b^{\text{th}}$ moment bound for zeta and the Cauchy-Schwarz inequality to obtain a bound of the right order of magnitude.
To handle larger shifts, stronger techniques are needed however since $\zeta(\tfrac{1}{2}+ i t + i h_1)$ and $\zeta(\tfrac{1}{2}+ i t + i h_2)$ decouple.
We will adapt the method of Heap, Radziwiłł, and Soundararajan \cite{HRS} and use twisted moments to efficiently estimate the correlation.
Throughout this section we will select the following parameters:
Denote by $\log_j$ the $j$-fold iterated logarithm, and we will take $\ell$ to be the largest integer so that $\log_\ell T \geq 10^4$.
Let $T_0 = e^2$ and for $1\leq j \leq \ell$ take
\[
T_j = \exp\left(\frac{\log T}{(\log_{j+1} T)^2} \right)
\]
and $K_j = 250P_j$ for $1\leq j \leq \ell$.

The first step will be to bound the integrand in (\ref{eqn:MoMTensor}) by products of integral powers of zeta and Dirichlet polynomials following \cite{HRS}.

\begin{prop}\label{prop:interpolation}
For $0 \leq \b \leq 1$ and $s,w \in \C$
\begin{align*}
    |\zeta(s)|^{2\b} |\zeta(w)|^{2\b} &\ll  |\zeta(s)|^2|\zeta(w)|^2 \prod_{1\leq j \leq \ell} |\n_{j}(s;\b-1)|^2|\n_{j}(w;\b-1)|^2 \\
    +  \prod_{1\leq j \leq \ell} |\n_{j}(s;\b)|^2|\n_{j}(w;\b)&|^2
    + \sum_{1\leq v \leq \ell} \bigg(|\zeta(s)|^2|\zeta(w)|^2 \prod_{1\leq j <v} |\n_{j}(s;\b-1)|^2|\n_{j}(w;\b-1)|^2 \\
    + \prod_{1\leq j < v} &|\n_{j}(s;\b)|^2|\n_{j}(w;\b)|^2\bigg)\left(\Bigg|\frac{\p_v(s) }{50 P_v}\Bigg|^{2 \lceil 50 P_v \rceil} + \Bigg|\frac{\p_v(w) }{50 P_v}\Bigg|^{2 \lceil 50 P_v \rceil}\right).
\end{align*}
\end{prop}

\begin{proof}
Take $1\leq v \leq \ell$ to be the smallest index such that either  $|\p_v(s)| > 50 P_v$ or $|\p_v(w)| > 50 P_v$, and set $v = \ell + 1$ if no such index exists.
Applying Young's inequality $ab \leq a^p/p + b^q/q$ for conjugate exponents $p = 1/\b$, $q = 1/(1-\b)$ gives
\[
|\zeta(s)|^{2\b}|\zeta(w)|^{2\b} \prod_{1\leq j < v} e^{2(1-\b)(\Re \p_j(s)+\Re \p_j(w))} \ll  |\zeta(s)|^{2}|\zeta(w)|^{2} + \prod_{1\leq j < v} e^{2(\Re \p_j(s)+\Re \p_j(w))},
\]
hence
\[
|\zeta(s)|^{2\b}|\zeta(w)|^{2\b} \ll  |\zeta(s)|^{2}|\zeta(w)|^{2} \prod_{1\leq j < v} e^{2(\b-1)(\Re \p_j(s)+\Re \p_j(w))} + \prod_{1\leq j < v} e^{2\b (\Re \p_j(s)+\Re \p_j(w))}.
\]
The next step is to expand the exponentials into Dirichlet polynomials.
Since $|\p_j(s)| \leq 50 P_j$ and $|\p_j(w)| \leq 50 P_j$ for $j < v$, lemma 1 of \cite{HRS} furnishes the bound
\begin{align*}
|\zeta(s)|^{2\b}|\zeta(w)|^{2\b} &\ll  |\zeta(s)|^{2}|\zeta(w)|^{2} \prod_{1\leq j < v} |\n_j(s;\b-1)|^2 |\n_j(w;\b-1)|^2 (1-e^{-P_j})^{-2} \\ &+\prod_{1\leq j < v} |\n_j(s;\b-1)|^2 |\n_j(w;\b-1)|^2 (1-e^{-P_j})^{-2}.
\end{align*}
To conclude, note that when $1\leq v \leq \ell$ we may multiply the right-hand side by
\[
\Bigg|\frac{\p_v(s) }{50 P_v}\Bigg|^{2 \lceil 50 P_v \rceil} + \Bigg|\frac{\p_v(w) }{50 P_v}\Bigg|^{2 \lceil 50 P_v \rceil} \geq 1.
\]
Since $\prod_{ 1\leq j \leq \ell} (1-e^{-P_j})^{-2} \leq 4$, the claim follows by summing over $1\leq v \leq \ell + 1$.

\end{proof}

Therefore proving Theorem \ref{thm:mainUpperBounds} is now reduced to a handful of moment calculations. 
To simplify the notation, we will write
\[
\n_{h_1,h_2,j}(s;\b) := \n_{j}(s + i h_1;\b)\n_j(s+ i h_2;\b).
\]

\begin{prop}\label{prop:upperBoundTwistedFourth}
If $\b \leq 1$ and  $|h_1|,|h_2| \leq (\log T)^\t$, then for large $T$
\begin{align*}
\frac{1}{T}\int_T^{2T}  | \zeta(\tfrac{1}{2} + i t + i &h_1)\zeta(\tfrac{1}{2} + i t + i h_2)|^2 \prod_{1\leq j \leq \ell} |\n_{h_1,h_2,j}(\tfrac{1}{2} + i t;\b-1)|^2 dt \stepcounter{equation}\tag{\theequation}\label{eqn:twisted2Nint}\\
&\ll (\log T)^{2\b^2} |\zeta(1 + 1/\log T + i(h_1- h_2))|^{2\b^2}
\end{align*}
and for $1\leq v \leq \ell$, $k \in \{1,2\}$ and $0 \leq r \leq \lceil 50 P_v \rceil$
\begin{align*} 
\frac{1}{T}\int_T^{2T}  |\zeta(\tfrac{1}{2} &+ i t + i h_1)\zeta(\tfrac{1}{2} + i t + i h_2)|^2 \prod_{1\leq j < v} |\n_{h_1,h_2,j}(\tfrac{1}{2} + i t;\b-1)|^2 |\p_v(\tfrac{1}{2} + i (t + h_k))|^{2r} dt \\
&\ll (\log T)^2 \left(\frac{(\log T_{v-1})^2}{\log T} \right)^{2\b^2-2}  |\zeta(1 + 1/\log T + i(h_1- h_2))|^{2\b^2} \left(18^r r! P_v^r \exp(P_v)\right).
\end{align*}
\end{prop}

\begin{prop}\label{prop:upperBoundDPMV}
If $\b \leq 1$ and $|h_1|,|h_2| \leq (\log T)^\t$, then for large $T$
\begin{equation}\label{eqn:DPUpperBound}
\frac{1}{T}\int_T^{2T}  \prod_{1\leq j \leq \ell} |\n_{h_1,h_2,j}(\tfrac{1}{2} + i t;\b)|^2 dt \ll (\log T)^{2\b^2} |\zeta(1 + 1/\log T + i(h_1- h_2))|^{2\b^2}
\end{equation}
and for $1\leq v \leq \ell$, $k \in \{1,2\}$ and $0 \leq r \leq \lceil 50 P_v \rceil$
\begin{align*} 
\frac{1}{T}\int_T^{2T}  &\prod_{1\leq j < v} |\n_{h_1,h_2,j}(\tfrac{1}{2} + i t;\b)|^2 |\p_v(\tfrac{1}{2} + i (t + h_k))|^{2r} dt \\ 
&\ll \left(\log T \right)^{2\b^2} |\zeta(1 + 1/\log T + i(h_1- h_2))|^{2\b^2} \left(r! P_v^r \right).
\end{align*}
\end{prop}

\noindent
Before embarking on the proofs of these propositions, we first show how to deduce the upper bound in Theorem \ref{thm:shiftedMoments} and then Theorem \ref{thm:mainUpperBounds}.

\begin{proof}[Proof of Theorem \ref{thm:shiftedMoments}: upper bound case.]
Propositions \ref{prop:interpolation}, \ref{prop:upperBoundTwistedFourth}, and \ref{prop:upperBoundDPMV} imply
\begin{align*}
\frac{1}{T} \int_T^{2T} |\zeta(\tfrac{1}{2} + &i t + ih_1)\zeta(\tfrac{1}{2} + i t + ih_2)|^{2\b} ~dt \\
\ll (\log T)^{2\b^2}& |\zeta(1 + 1/\log T + i (h_1 -  h_2))|^{2\b^2} \\
\times\Bigg[1 + \sum_{1\leq v \leq \ell} 
\Bigg(\left(\frac{\log T}{\log_{v-1 }T}\right)^4
&\frac{18^{\lceil 50 P_v\rceil} \lceil 50 P_v\rceil! P_v^{\lceil 50 P_v\rceil} \exp(P_v)}{(50P_v)^{2\lceil 50 P_v\rceil}} 
+ \frac{\lceil 50 P_v\rceil! P_v^{\lceil 50 P_v\rceil}}{(50P_v)^{2\lceil 50 P_v\rceil}} \Bigg) \Bigg] .
\end{align*}
A quick calculation shows the sum over $v$ is $O(1)$, so the result follows.

\end{proof}

\begin{proof}[Proof of Theorem \ref{thm:mainUpperBounds}]
Split the range of integration in (\ref{eqn:MoMTensor}) depending on whether $|h_1-h_2| < 1/\log T$ or $|h_1-h_2| \geq 1/\log T$.
By the Laurent expansion for zeta near 1, the contribution of the region where $ |h_1-h_2| < 1/\log T$ to the integral (\ref{eqn:MoMTensor}) is $\ll (\log T)^{4\b^2 + \t - 1}$, which is admissible.
To handle the integral when $|h_1 - h_2| \geq 1/\log T$, we treat the cases $\t \leq 0$ and $\t > 0$ separately.
When $\t \leq 0$, applying the upper bound in Theorem \ref{thm:shiftedMoments} the Laurent expansion for zeta gives the bound
\[
\MoM_T(2,\b) \ll (\log T)^{2\b^2 + \t}\int_{1/\log T}^{2(\log T)^\t} \frac{dh}{h^{2\b^2}}.
\]
Therefore when $\t \leq 0$ Theorem \ref{thm:mainUpperBounds} immediately follows. 
When $\t > 0$, standard estimates for the moments of zeta to the right of the one line \cite[Theorem 7.1]{Titchmarsh} now show that 
\[
\MoM_T(2,\b) \ll (\log T)^{2\b^2 + \t}\int_{1/\log T}^{1} \frac{dh}{h^{2\b^2}} + (\log T)^{2\b^2 + 2\t}.
\]
Note the implicit constant is absolute because $\b \leq 1$.

\end{proof}

\subsection*{Proof of Proposition \ref{prop:upperBoundTwistedFourth}}

We will apply on Theorem \ref{thm:twisted4th} with $\bm{\a} = (i h_1, i h_2, -i h_1, - i h_2)$ and $w(t)$ a smooth majorant of $1_{t\in[1,2]}$ to the Dirichlet polynomials  
\begin{equation}\label{eqn:upperBdDP1}
\prod_{1\leq j \leq \ell} \n_{h_1,h_2,j}(s;\b - 1)
\end{equation}
and
\begin{equation}\label{eqn:upperBdDP2}
\prod_{1\leq j \leq v} \n_{h_1,h_2,j}(s;\b - 1) \p_v(\tfrac{1}{2}  + i (t+ h_k))^r
\end{equation}
for $k\in\{1,2\}$.
We may choose circular contours for the $z_j$ so that
\[
A(\bm{z}) \ll (\log T)^2 |\zeta(1+ 1/\log 
T + i(h_1 - h_2))|^{2},
\]
\[
\Delta(z_1,z_2,-z_3,-z_4)^2 \ll (\log T)^{-4} |h_1 - h_2 + 1/\log T|^8.
\]
Because $\Re z \ll 1/\log T$ it readily follows that
\[
\int_\R Y_{\bm{\a},t} ~ Y_{-\pi{\bm{z}},t} ~ w(t/T) ~ dt \ll T.
\]
Therefore, by multiplicativity we are left with the task of bounding the  sums
\[
F_{h_1,h_2,j}(\bm{z}) = \sum_{n,m\leq T^\eta} \frac{a_{h_1,h_2,j}(n) \overline{a_{h_1,h_2,j}(m)}}{[n,m]} B_{\bm{z}}\left(\frac{n}{(n,m)}\right) B_{\pi\bm{z}}\left(\frac{m}{(n,m)}\right)
\]
and
\[
F_{v,r,h_k}(\bm{z}) = \sum_{n,m\leq T^\eta} \frac{b_{v,r,h_k}(n) \overline{b_{v,r,h_k}(m)}}{[n,m]} B_{\bm{z}}\left(\frac{n}{(n,m)}\right) B_{\pi\bm{z}}\left(\frac{m}{(n,m)}\right),
\]
uniformly for $|z_j - \a_j| \ll 1/\log T$, where $a_{h_1,h_2,j}(n)$ are the coefficients of the Dirichlet polynomials $\n_{h_1,h_2,j}(s)$, and $b_{v,r,h_k}(n)$ are the coefficients of $\p_v(s+ i h_k)^{2r}$.

First we estimate $F_{h_1,h_2,j}(\bm{z})$. Note we can write
\[
\n_{h_1,h_2,j}(\tfrac{1}{2} + i t;\b - 1) = \sum_{\substack{p\mid n \Rightarrow p \in (T_{j-1},T_j] \\ \Omega(n) \leq K_j}} \frac{1}{n^{1/2 + i t}} \sum_{c d = n } \frac{(\b-1)^{\Omega(c) + \Omega(d)} g(c) g(d)}{c^{i h_1} d^{i h_2}},
\]
so $a_{h_1,h_2,j}(p) = (\b-1)(p^{-i h_1} + p^{-ih_2}) = (\b-1) \sigma_{i h_1, i h_2}(p)$ for $p\in(T_{j-1},T_j]$. Since $\b \leq 1$ it also follows that $|a_{h_1, h_2,j}(n)| \leq d(n)$, where $d$ is the divisor function. 
We will require the following estimates for $B_{\bm{z}}(n)$.

\begin{lemma}\label{lem:BEstimate}
For $m\geq 1$
\begin{align*}
 B_{\bm{\a}}(p^m) =  \frac{\sig_{\a_3,\a_4}(p^m) - \sig_{\a_3,\a_4}(p^{m-1}) p^{-1-\a_3-\a_4}(p^{-\a_1}+p^{-\a_2}) + \sig_{\a_3,\a_4}(p^{m-2})p^{-2-\a_1-\a_2-2\a_3-2\a_4} }{1-p^{-2-\a_1-\a_2-\a_3-\a_4}},
\end{align*}
where by convention we set $\sig_{\a_3,\a_4}(p^{-1}) = 0$.
Furthermore, for integers $n$ composed of primes at most $T^{10^{-8}}$ and $\Re \a_j \ll 1/\log T$ for each $j$
\[
|B_{\bm{\a}}(n)| \ll d_3(n),
\]
where $d_3$ is the ternary divisor function.
\end{lemma}
\begin{proof}
The first formula follows from lemma 6.9 of \cite{HYTwisted4} and that 
\[
\sigma_{z_1,z_2}(p^m) = \frac{p^{-z_3(m+1)} - p^{-z_4(m+1)} }{p^{-z_3} - p^{-z_4}}.
\]
To prove the second bound, first note by assumption on the size of the shifts and the size of the primes $p$
\[
|\sigma_{\a_3,\a_4}(p^m)| \leq \sum_{ab = p^m} a^{\Re \a_3} b^{\Re \a_4} \leq (1+m)\left(1 + O\left(\frac{\log p}{\log T}\right)\right).
\]
Whence
\[
|B_{\bm{\a}}(p^m)| \leq d(p^m)\left(1 + O\left(\frac{\log p}{\log T} + \frac{1}{p}\right)\right).
\]
It now follows by assumption on $n$ that
\[
|B_{\bm{\a}}(n)| \ll d(n) \prod_{p^m \| n} \left(1 + O\left(\frac{m\log p}{\log T}\right) \right) \left(1 + O\left(\frac{1}{p}\right)\right) \ll d_3(n)
\]
where we have used $\prod_{p^m \| n} \left(1 + O\left(\frac{m\log p}{\log T}\right) \right) \ll \exp(\log n/\log T)$ and $\prod_{p \mid n} \left(1 + O\left(\frac{1}{p}\right)\right) \ll (3/2)^{\omega(n)}$, where $\omega(n)$ is the number of distinct prime divisors of $n$.

\end{proof}

The following lemma is almost identical to the proof of Lemma 24 of \cite{FHKUpper}. 
The proof is a straightforward application of Cauchy's theorem.
The only necessary modification is to use Lemma \ref{lem:BEstimate} to bound $B_{\bm{z}}(p)$ because here $z$ need not have small imaginary part.

\begin{lemma}\label{lem:BContinuity}
Let $m \geq 1$ be an integer and  $\bm{z} = (z_1,z_2,z_3,z_4)$,  $\bm{w} = (w_1,w_2,w_3,w_4)$ vectors such that $|w_j - z_j| \ll 1/\log T$ for all $j$.
Then
\[
|B_{\bm{z}}(p) - B_{\bm{w}}(p)| \ll \frac{ \log p}{\log T}.
\]
\end{lemma}

We will now bound $F_{h_1,h_2,j}(\bm{z})$ by a product over primes.
To accomplish this, we will drop the terms with $\Omega(m),\Omega(n)\leq K_j$ and use Rankin's trick to show the error incurred is negligible.
First note that $\exp(\Omega(n) + \Omega(n) - K_j) \geq 1$ when either $\Omega(n)$ or $\Omega(m)$ exceed $K_j$.
Now using Lemma \ref{lem:BEstimate} and that $|\b-1|\leq 1$, we see that adding in the terms of $F_{h_1,h_2,j}(\bm{z})$ with $\Omega(m) > K_j$ or $\Omega(n) > K_j$ contributes an error of at most
\begin{align*}
    &\ll e^{-K_j} \sum_{p| m,n \Rightarrow p \in (T_{j-1},T_j]} \frac{(|\b-1|e)^{\Omega(n)+ \Omega(m)}}{[n,m]} d_3(n)d_3(m)\\
    &\ll e^{-K_j} \prod_{T_{j-1} < p \leq T_j} \left(1 + \frac{6e+ 9e^2}{p} + O\left(\frac{1}{p^2}\right) \right) \ll e^{-70 P_j}
\end{align*}
Therefore upon dropping the assumption that $\Omega(m),\Omega(n)\leq K_j$ and using Lemmas \ref{lem:BEstimate}, \ref{lem:BContinuity}, and that $|a_{h_1,h_2}(n)| \leq d(n)$, we may bound $F_{h_1,h_2,j}(\bm{z})$ by
\begin{align*}
 \prod_{T_{j-1} < p \leq T_j} \bigg(1 + \frac{a_{h_1,h_2,j}(p)B_{ih_1,ih_2,-ih_1,-ih_2}(p) + a_{-h_1,-h_2,j}(p)B_{-ih_1,-ih_2,ih_1,ih_2}(p) + |a_{h_1,h_2}(p)|^2}{p} \\
 + O\left(\frac{\log p}{p \log T} +  \frac{1}{p^2} \right)\bigg) + O(e^{-70 P_j}) \qquad \qquad \qquad
\end{align*} 
By Lemma \ref{lem:BEstimate}, the numerator of the second term in parentheses equals
\[
\frac{(\b-1)(\b-1 + p + \b p)}{1+ p} (2 + 2\cos((h_1-h_2)\log p)) = (2\b^2 - 2)(1 + \cos((h_1-h_2)\log p)) + O\left(\frac{1}{p}\right).
\]
Therefore, we must control the products
\begin{align*}
 \prod_{T_{j-1} < p \leq T_j} \bigg(1 + \frac{ (2\b^2 - 2)(1 + \cos((h_1-h_2)\log p)) }{p} + O\left(\frac{\log p}{p \log T} +  \frac{1}{p^2} \right)\bigg) 
\end{align*} 
To this end, we will need the following special case of lemma 3.2 of \cite{Koukoulopoulos}.
\begin{lemma}\label{lem:cosPrimeSum}
Given $h \in \R$ and $X \geq 2$ 
\begin{align*}
\sum_{p\leq X} &\frac{\cos(h \log p)}{p}  = \log|\zeta(1 + 1/\log X + i h)| + O(1).
\end{align*}
\end{lemma}
\begin{rmk}
While similar results have appeared in previous literature, this formulation is due to Granville and Soundararajan \cite{GS}.
\end{rmk}
\begin{proof}
Because the Euler product of $\zeta(s)$ is convergent for $\Re s > 1$ we may write
\begin{align*}
\log|\zeta(1 + 1/\log X + i h)| &= \Re \left(\sum_{p} \frac{1}{p^{1+1/\log X + i h}} + \sum_p \sum_{m\geq 2} \frac{1}{mp^{m(1+1/\log X + i h)}}\right) \\
&= \sum_{p} \frac{\cos(h \log p)}{p^{1+1/\log X}} + O(1).
\end{align*}
The primes larger than $X$ contribute at most
\[
\sum_{p > X} \frac{1}{p^{1+ 1/\log X}}  \ll 1. 
\]
For $p \leq X$ the mean value theorem gives  $p^{1/\log X} = 1 + O(\log p /\log X)$, so the claim follows from Merten's first estimate $\sum_{p\leq X} \frac{\log p}{p} = \log X + O(1).$
\end{proof}

Therefore uniformly for $|z_j - \a_j|\ll 1/\log T$, the product $\prod_{1\leq j < v}  F_{h_1,h_2,j}(\bm{z})$ is of order at most
\begin{align*}
\prod_{1\leq j < v} \Bigg( \prod_{T_{j-1} < p \leq T_j} \bigg(1 + \frac{ (2\b^2 - 2)(1 + \cos((h_1-h_2)\log p)) }{p} &+ O\left(\frac{\log p}{p \log T} +  \frac{1}{p^2} \right)\bigg) + O(e^{-70 P_j}) \Bigg) \\ 
\ll (\log T_{v-1})^{2\b^2 - 2} \left(\frac{\log T_{v-1}}{\log T_{\ell}} \right)^{2\b^2 - 2}|\zeta(1 &+ 1/\log T_\ell + i(h_1-h_2))|^{2\b^2 - 2}
\\
\ll \left(\frac{(\log T_{v-1})^2}{\log T_{\ell}} \right)^{2\b^2 - 2} |\zeta(1 + 1/\log T_\ell &+ i(h_1-h_2))|^{2\b^2 - 2}.
\end{align*}
Since $\log T_\ell \asymp \log T$, this proves the first bound in Proposition \ref{prop:upperBoundTwistedFourth}.
All that remains is to estimate the $F_{v,r,h_k}(\bm{z})$. 
Note that $|d_3(n)|\ll 3^{\Omega(n)}$ and we can trivially estimate $|n^{-ih_k}| \leq 1$, so
\[
|F_{v,r,h_k}(\bm{z})| \leq 9^r \sum_{\substack{p|m,n \Rightarrow p\in (T_{v-1},T_v]\\ \Omega(n) = \Omega(m) = r}} \frac{r!^2 g(n)g(m)}{[n,m]}.
\]
In the proof of proposition 3 of \cite{HRS}, it was shown that the right-hand side of is at most $18^r r! P_v^r \exp (P_v)$, which concludes the proof of the second bound of Proposition \ref{prop:upperBoundTwistedFourth}.

\subsection*{Proof of Proposition \ref{prop:upperBoundDPMV}}

To handle the first bound in Proposition \ref{prop:upperBoundDPMV},  we will use the mean value theorem for Dirichlet polynomials \cite[Theorem 9.1]{IK} in place of Theorem \ref{thm:twisted4th}.
We will abuse of notation and also denote the coefficients of $N_{h_1,h_2,j}(s;\b)$ by $a_{h_1,h_2,j}(n)$ now with $\b$ in place of $\b-1$. Therefore
\begin{align*}
\frac{1}{T}\int_T^{2T}  \prod_{1\leq j \leq \ell} |\n_{h_1,h_2,j}(\tfrac{1}{2} + i t;\b)|^2 dt \ll \prod_{1\leq j \leq \ell} \sum_{\substack{p\mid n \Rightarrow p \in (T_{j-1},T_j] \\ \Omega_j(n) \leq K_j \forall j}} \frac{ |a_{h_1,h_2,j}(n)|^2}{n}.
\end{align*}
The proof is similar to the proof of Proposition \ref{prop:upperBoundTwistedFourth}, but simpler because there is no need to use Rankin's trick because all the terms we will discard are positive. 
The argument in the proof of Proposition \ref{prop:upperBoundTwistedFourth} now gives the bound
\begin{align*}
\ll \prod_{1\leq j \leq \ell}  \prod_{T_{j-1} < p \leq T_j} &\left(1 + \frac{ 2\b^2 (1 + \cos((h_1-h_2)\log p)) }{p} + O\left(\frac{\log p}{p \log T} +  \frac{1}{p^2} \right) \right)  \\ 
&\ll (\log T)^{2\b^2} |\zeta(1 + 1/\log T + i(h_1-h_2))|^{2\b^2}.
\end{align*}
To prove the second bound of Proposition \ref{prop:upperBoundDPMV}, note that the mean value theorem for Dirichlet polynomials and the same reasoning as the proof of Proposition \ref{prop:upperBoundTwistedFourth}  now gives a bound of
\begin{align*}
\ll \prod_{1\leq j < v} \prod_{T_{j-1} < p \leq T_j} &\left(1 + \frac{ 2\b^2 (1 + \cos((h_1-h_2)\log p)) }{p} + O\left(\frac{\log p}{p \log T} +  \frac{1}{p^2} \right)\right) \sum_{\substack{p| n \Rightarrow p \in (T_{v-1}, T_v]\\ \Omega(n) = r}} \frac{(r! g(n))^2}{n}
\\
&\ll  \left(\log T \right)^{2\b^2} |\zeta(1 + 1/\log T+ i(h_1-h_2))|^{2\b^2} (r! P_v^r).
\end{align*}
Here we have used that the sum over $n$ above is bounded by $r! P_v^r$, which is a porism of proposition 2 of \cite{HRS}.
This concludes the proofs of Propositions \ref{prop:upperBoundTwistedFourth} and  \ref{prop:upperBoundDPMV}, so Theorem \ref{thm:shiftedMoments} hence Theorem \ref{thm:mainUpperBounds} follows. \qed

\section{Lower Bounds}\label{sec:lowerBounds}

Theorem \ref{thm:mainLowerBound} follows by integrating the lower bound given by Theorem \ref{thm:shiftedMoments} over the range $0 \leq |h_1 - h_2|  \leq 2(\log T)^\t$ using the Laurent expansion of $\zeta$ and standard moment estimates for $\zeta$ to the right of the one line.
We will first prove Theorem \ref{thm:shiftedMoments}  in the case $\b \leq 1$.

\subsection{Proof of Theorem \ref{thm:shiftedMoments} for $\b \leq 1$}

Throughout this subsection, we will let $T_0 = e^2$, let $\ell$ be the largest integer such that $\log_\ell T \geq 10^4,$ and for $1\leq j \leq \ell$ set
\[
T_j = \exp\left(\frac{\b \log T}{(\log_j T)^2}\right)
\]
and $K_j = 250 P_j$.
With this choice of parameters, $\n(s;\a)$ is a Dirichlet polynomial of length at most $ T^{\b/18}$.
Inspired by the method of Heap and Soundararajan \cite{HS}, our key inequality is the following consequence of Hölder's inequality:

\begin{align} \label{eqn:InterpolB<1} 
\Bigg|\int_\R \zeta(\tfrac{1}{2} + it + i h_1) \zeta(\tfrac{1}{2} + it &+ i h_2) \n(\tfrac{1}{2}+ i t + i h_1;\b-1) \n(\tfrac{1}{2}+ i t + i h_2;\b-1) \\
&\quad\times \n(\tfrac{1}{2} - i t - i h_1;\b) \n(\tfrac{1}{2}- i t - i h_2;\b) w(t/T)~ dt  \Bigg| \nonumber \\
\leq \bigg(\int_\R &|\zeta(\tfrac{1}{2} + it + i h_1)  \zeta(\tfrac{1}{2} + it + i h_2)|^{2\b} w(t/T) ~dt\bigg)^{\tfrac{1}{2}} \nonumber  \\
\times\bigg(\int_\R |\zeta(\tfrac{1}{2} + it + i h_1) &\zeta(\tfrac{1}{2} + it + i h_2)|^2 \nonumber  \\
\times |\n(\tfrac{1}{2}+ i t &+ i h_1;\b-1)\n(\tfrac{1}{2}+ i t + i h_2;\b-1)|^2 w(t/T) ~ dt\bigg)^{\tfrac{1-\b}{2}} \nonumber \\ 
\times\bigg(\int_\R |\n(\tfrac{1}{2}+ i t + i h_1;\b-1)&\n(\tfrac{1}{2} + i t + i h_2;\b-1)|^2 \nonumber\\
&\times |\n(\tfrac{1}{2}+ i t + i h_1;\b)\n(\tfrac{1}{2}+ i t + i h_2;\b)|^{2/\b} w(t/T) ~ dt \bigg)^{\tfrac{\b}{2}}, \nonumber
\end{align}
 
\noindent 
where $0\leq w(t)\leq 1$ is a smooth function supported on the interval $[1.1, 1.9]$ with $w(t) = 1$ for $t\in [1.2,1.8]$, say.
Note that Proposition \ref{prop:upperBoundTwistedFourth} gives the upper bound  
\begin{align*}
\int_\R |\zeta(\tfrac{1}{2} + it + i h_1) \zeta(\tfrac{1}{2} + it + i h_2)|^2  &|\n(\tfrac{1}{2}+ i t + i h_1;\b-1)\n(\tfrac{1}{2}+ i t + i h_2;\b-1)|^2 w(t/T) ~dt
\\
&\ll T(\log T)^{2\b^2} |\zeta(1+1/\log T + i(h_1-h_2))|^{2\b^2}
\end{align*}
so we are left with two moment computations. 

\begin{rmk}
It is also natural to try to prove Theorem \ref{thm:mainLowerBound} by computing the twisted  moment
\begin{align*}
\int_\R \zeta(\tfrac{1}{2} + it + i h_1) \zeta(\tfrac{1}{2} - it - i h_2) &\n(\tfrac{1}{2}+ i t + i h_1;\b-1) \n(\tfrac{1}{2}+ i t + i h_2;\b) \\
&\times \n(\tfrac{1}{2} - i t - i h_1;\b) \n(\tfrac{1}{2}- i t - i h_2;\b-1) w(t/T)~ dt.
\end{align*}
While this would allow us to use Theorem \ref{thm:twisted2nd},
it seems more difficult to evaluate this integral asymptotically.
\end{rmk}

Before carrying out these computations, it will be convenient to introduce some notation.
For $1\leq j \leq \ell$, we set
\[
A_{h_1,h_2,j}(\tfrac{1}{2} + it) :=   \n_{j}(\tfrac{1}{2} + it + i h_1;\b-1)\n_j(\tfrac{1}{2} + it + i h_2;\b-1),
\]
\[
B_{h_1,h_2,j}(\tfrac{1}{2} + it) :=   \n_{j}(\tfrac{1}{2} + it + i h_1;\b)\n_j(\tfrac{1}{2} + it + i h_2;\b).
\]
These are Dirichlet polynomials supported on integers $n$  such that $\Omega(n) \leq K_j$ and $n$ is composed of primes $p\in (T_{j-1},T_j]$ with coefficients
\[
a_{h_1,h_2,j}(n) :=  \sum_{cd = n} \frac{(\b-1)^{\Omega(c) + \Omega(d)}  g(c)g(d)}{c^{i h_1} d^{i h_2}},
\]
\[
b_{h_1,h_2,j}(n) := \sum_{cd = n} \frac{\b^{\Omega(c) + \Omega(d)} g(c)g(d)}{c^{i h_1} d^{i h_2}}.
\]
Additionally, for $a = 1,2$ define 
\[
M_a(t) = |\n(\tfrac{1}{2}+ i t + i h_a;\b-1)| |\n(\tfrac{1}{2}+ i t + i h_a;\b)|^{1/\b}.
\]
The lower bound of Theorem \ref{thm:shiftedMoments} will follow from the following two propositions.

\begin{prop}\label{prop:lowerTwistedB<1}
Uniformly in $|h_1|, |h_2| \leq (\log T)^\t$
\begin{align*}
\Bigg|\frac{1}{T}\int_\R \zeta(\tfrac{1}{2} + it + i h_1) \zeta(\tfrac{1}{2} - it - i h_2) \prod_{1\leq j\leq \ell} &A_{h_1,h_2,j}(\tfrac{1}{2}+i t)\overline{B_{h_1,h_2,j}(\tfrac{1}{2}+i t)} w(t/T) ~dt ~\Bigg| \\
&\gg_\b (\log T)^{2\b^2} |\zeta(1+1/\log T +i(h_1-h_2))|^{2\b^2}.
\end{align*}
\end{prop}

\begin{prop}\label{prop:lowerDPMVB<1}
Uniformly in $|h_1|, |h_2| \leq (\log T)^\t$ 
\[
\frac{1}{T}\int_\R M_1 M_2(t)^2 w(t/T) dt \ll_\b (\log T)^{2\b^2} |\zeta(1+1/\log T +i(h_1-h_2))|^{2\b^2}.
\]
\end{prop}

\subsubsection*{Proof of Proposition \ref{prop:lowerTwistedB<1}}

The task at hand is to evaluate the integral
\begin{align*}
\I(\b,h_1,h_2) =\int_\R \zeta(\tfrac{1}{2} + it + i h_1) \zeta(\tfrac{1}{2} + it &+ i h_2) \n(\tfrac{1}{2}+ i t + i h_1;\b-1) \n(\tfrac{1}{2}+ i t + i h_2;\b-1) \\
&\times \n(\tfrac{1}{2} - i t - i h_1;\b) \n(\tfrac{1}{2}- i t - i h_2;\b) w(t/T)~ dt.   
\end{align*}
We will accomplish this with an appropriate approximate functional equation.
Set
\[
\wt{g}_{\a_1,\a_2}(s,t) = \pi^{-s} \frac{\Gamma\left(\frac{\tfrac{1}{2} + \a_1 + s + i t}{2}\right)\Gamma\left(\frac{\tfrac{1}{2} + \a_2 + s + i t}{2}\right)}{\Gamma\left(\frac{\tfrac{1}{2} + \a_1 + i t}{2}\right)\Gamma\left(\frac{\tfrac{1}{2} + \a_2  + i t}{2}\right)},
\]
\[
\wt{X}_{\a_1,\a_2,t} =  \lam(\tfrac{1}{2} + \a_1  + i t) \lam(\tfrac{1}{2} + \a_2 + i t),
\]
and
\[
\wt{V}_{\a_1,\a_2}(x,t) = \frac{1}{2\pi i} \int_{(1)} \frac{e^{s^2}}{s} \wt{g}_{\a_1,\a_2}(s,t)x^{-s} ~ ds,
\]
which satisfies the decay estimate
\[
t^j \frac{\del^j}{\del t^j} \wt{V}_{\a_1,\a_2}(x,t)  \ll_{A,j} (1 + |x|/t)^{-A}.
\]

\begin{lemma}[\cite{IK} Theorem 5.3]\label{lem:AFE2nd}
Let $t\in [T,2T]$, $\Re \a_j \ll 1/ \log T$ and $|\Im \a_j| \leq (\log T)^\theta$ for $j = 1,2$ and $\theta > -1$. Then for all $A > 0$ 
\begin{align*}
    \zeta(\tfrac{1}{2} + \a_1 + i t) \zeta(\tfrac{1}{2} + \a_2 + i t) = \sum_{m,n} \frac{1}{m^{1/2 + \a_1} n^{1/2 + \a_2}} \left(m n \right)^{- i t} \wt{V}_{\a_1,\a_2}(mn,t)\\
    + \wt{X}_{\a_1,\a_2, t} \sum_{m,n} \frac{1}{m^{1/2 - \a_2} n^{1/2 - \a_1}} \left(mn\right)^{i t} \wt{V}_{-\a_2,-\a_1}(mn,t) + O_A((1 + t)^{-A}).
\end{align*}
\end{lemma}

Therefore, up to a negligible error, that $I(\b,h_1,h_2)$ is given by 
\begin{align} \label{eqn:lowerBound2Term}
&\prod_{j\leq \ell} \sum_{h,k,m,n} \frac{a_{h_1,h_2,j}(h)\overline{b_{h_1,h_2,j}(k)}}{m^{h_1}n^{h_2}\sqrt{hkmn}}  \int_\R  \left(\frac{hmn}{k}\right)^{-it} \wt{V}_{h_1,h_2}(mn,t)  w(t/T)~dt \\
+  
\prod_{j\leq \ell}&\sum_{h,k,m,n} \frac{a_{h_1,h_2,j}(h)\overline{b_{h_1,h_2,j}(k)}}{m^{-h_2}n^{-h_1}\sqrt{hkmn}} \int_\R \wt{X}_{h_1,h_2,t} \left(\frac{h}{kmn}\right)^{-it} \wt{V}_{-h_2,-h_1}(mn,t) w(t/T)~dt. \nonumber
\end{align}
Using the decay of $\widetilde{V}_{\a_1,\a_2}$, we can discard the contribution of terms with $hmn \neq k$ and $h\neq kmn$ in the first and second lines of (\ref{eqn:lowerBound2Term}) respectively. 
Next we can parameterize the diagonal sums in the first and second lines by $\sum_{h|k, mn = k/h}$ and $\sum_{k|h, mn = h/k}$ respectively.
Then after using the definition of $\wt{V}_{\a,\b}$, shifting the contour, and recognizing the sum over $m$ and $n$ as a familiar Dirichlet convolution we find
\begin{align}\label{eqn:lowerBound2TermDiag}
\I(\b, h_1,h_2) &= \prod_{j\leq \ell} \sum_{h|k} \frac{a_{h_1,h_2,j}(h)\overline{b_{h_1,h_2,j}(k)}}{k} \sigma_{h_1,h_2}(k/h)  \int_\R  w(t/T)~dt \\
+  
\prod_{j\leq \ell}&\sum_{k|h} \frac{a_{h_1,h_2,j}(h)\overline{b_{h_1,h_2,j}(k)}}{h} \sigma_{-h_2,-h_1}(h/k) \int_\R \wt{X}_{h_1,h_2,t} ~ w(t/T)~dt + O(T^{1-\d}). \nonumber
\end{align}
for some $\d > 0$.

Due to the oscillating term $\wt{X}_{h_1,h_2,t}$ the second line in (\ref{eqn:lowerBound2TermDiag})  will not contribute to the leading order, and we will control this term first. 
Because $w(t/T)$ is supported on $[T,2T]$ and that $|h_1|,|h_2| \leq (\log T)^\t$, Stirling's approximation gives $\wt{X}_{h_1,h_2,t} = e^{-i f(t)}(1+ O(1/T))$ where
\[
f(t) = (h_1 + t) \log\left(\frac{t + h_1}{2\pi}\right) + (h_2 + t) \log\left(\frac{t + h_2}{2\pi}\right) -2t - h_1 - h_2 - \frac{\pi}{2}. 
\]
For $t\in [T,2T]$ note that $f'(t) \gg \log t$ and that $f''(t) \ll 1/t$.
Therefore, an integration by parts à la van der Corput gives the bound
\[
\int_\R \wt{X}_{h_1,h_2,t} ~ w(t/T)~dt \ll 1/\log T.
\]
Therefore because $\prod_{1\leq j\leq \ell} A_{h_1,h_2,j}(s)\overline{B_{h_1,h_2,j}(s)}$ is a Dirichlet polynomial of length $\leq T^{\b/9}$, we may absorb the second line of (\ref{eqn:lowerBound2TermDiag}) into the error term $O(T^{1-\d})$.

It remains to analyze the sums 
\[
S_j := \sum_{h|k} \frac{a_{h_1,h_2,j}(h)\overline{b_{h_1,h_2,j}(k)}}{k} \sigma_{h_1,h_2}(k/h) = \sum_{a,h} \frac{\sigma_{h_1,h_2}(a)  a_{h_1,h_2,j}(h) \overline{b_{h_1,h_2,j}(ah)}}{ah} 
\]
for $j \leq \ell$.
Here we are summing over all integers $a,h$ composed of primes in $(T_{j-1},T_j]$ such that $\Omega(a) + \Omega(h) \leq K_j$.
If we discard the latter condition on $a$ and $h$ via Rankin's trick, we incur an error at most 
\begin{align*}
e^{-K_j} &\prod_{p \in (T_{j-1},T_j]} \sum_{r,s \geq 0} \frac{|\sigma_{h_1,h_2}(p^r)  a_{h_1,h_2,j}(p^s)  \overline{b_{h_1,h_2,j}(p^{r+s})}|}{p^{r+s} } e^{\Omega(r) + \Omega(s)}\\
&\ll e^{-K_j} \prod_{p \in (T_{j-1},T_j]} \left(1 + \frac{2e  }{p} + O\left(\frac{1}{p^2}\right)\right) \ll e^{-100P_j},
\end{align*}
whence
\[
S_j = \prod_{p \in (T_{j-1},T_j]} \sum_{r,s \geq 0} \frac{\sigma_{h_1,h_2}(p^r)  a_{h_1,h_2,j}(p^s)  \overline{b_{h_1,h_2,j}(p^{r+s})}}{p^{r+s} } + O(e^{-100P_j}).
\]
Now since $ a_{h_1,h_2,j}(p) = (\b - 1)\sigma_{h_1,h_2}(p)$ and $b_{h_1,h_2,j}(p) = \b \sigma_{h_1,h_2}(p)$ it follows that 
\[
S_j = \prod_{p \in (T_{j-1},T_j]}\left(1 + \frac{2\b^2(1 + \cos((h_1-h_2) \log p))}{p} + O\left(\frac{1}{p^2}\right) \right) + O(e^{-100P_j}),
\]
so up to a power savings $\I(\b,h_1,h_2)$ is equal to 
\[
T\prod_{j\leq \ell} \left( \prod_{p \in (T_{j-1},T_j]}\left(1 + \frac{2\b^2(1 + \cos((h_1-h_2) \log p))}{p} + O\left(\frac{1}{p^2}\right) \right) + O(e^{-100P_j})\right) \|w\|_1.
\]
Using Lemma \ref{lem:cosPrimeSum} and that $\|w\|_1 > 0.8$ we conclude
\[
|\I(\b,h_1,h_2)| \gg_\b T (\log T)^{2\b^2} |\zeta(1 + 1/\log T_\ell + i (h_1-h_2))|,
\]
and Proposition \ref{prop:lowerTwistedB<1} follows. \qed

\subsubsection*{Proof of Proposition \ref{prop:lowerDPMVB<1}}

Since we only need an upper bound, we will bound the weight $w(t/T)$ by the characteristic function of $[T,2T]$.
Therefore, we must simply evaluate the mean value of the Dirichlet polynomial $M_1 M_2(t)$ on $[T,2T]$.
We will decompose each
\[
M_a(t) = \prod_{j\leq \ell} M_{a,j} (t)
\]
where
\[
M_{a,j} (t) := |\n_j(\tfrac{1}{2}+ i t + i h_a;\b-1)| |\n_j(\tfrac{1}{2}+ i t + i h_a;\b)|^{1/\b}.
\]
Following \cite{HS} we will now bound $|M_{1,j}M_{2,j}|^2$ by products of genuine Dirichlet polynomials.

\begin{lemma}\label{lem:LBInterp}
For $j\leq \ell$
\begin{align*}
|M_{1,j}(t) M_{2,j}(t)|^2 \leq  
|\n_j(\tfrac{1}{2} + i t &+ i h_1; \b) \n_j(\tfrac{1}{2}+ i t + i h_2; \b)|^2 (1+O(e^{-K_j/10})) \\
 + O\big(2^{2/\b} Q_{1,j}(t) |\n_j(\tfrac{1}{2}&+ i t + i h_2; \b)|^2(1+O(e^{-K_j/10})) \\
+ 2^{2/\b}Q_{2,j}(t) |\n_j(\tfrac{1}{2}+ &i t + i h_2; \b)|^2(1+O(e^{-K_j/10}))\\
 + 2^{4/\b} &Q_{1,j}(t) Q_{2,j}(t) \big),
\end{align*} 
where for $a = 1,2$
\[
Q_{a,j}(t) := \left(\frac{12|\p_j(\tfrac{1}{2} + i t + i h_a)|}{K_j}\right)^{2K_j} \sum_{r = 0}^{K_j/\b} \left(\frac{2e|\p_j(\tfrac{1}{2} + i t + i h_a)|}{r+1}\right)^{2r}
\]
\end{lemma}
\begin{proof}
The proof is nearly identical to the proof of lemma 1 of \cite{HS}.
If $|\p_j(\tfrac{1}{2} + i t + i h_a)|\leq K_j/10$, then we may write 
\begin{align*}
|M_{a,j}(t)|^2 &= \exp\left(2\b \Re \p_j(\tfrac{1}{2} + i t + i h_a)+\right) (1+O(e^{-K_j/10}))^2  \\
&= |\n_j(\tfrac{1}{2} + i t + i h_a; \b)|^2 (1+O(e^{-K_j/10})).
\end{align*} 
In the case $|\p_j(\tfrac{1}{2} + i t + i h_a)| > K_j/10$ for some $a$ then the proof of lemma 1 of \cite{HS} provides  the bounds 
\[
|\n_j(\tfrac{1}{2}+ i t + i h_a;\b-1)| \leq \left(\frac{12|\p_j(\tfrac{1}{2} + i t + i h_a)|}{K_j}\right)^{2K_j}
\]
and
\[
|\n_j(\tfrac{1}{2}+ i t + i h_a;\b)|^{1/\b} \ll 2^{2/\b} \sum_{r = 0}^{K_j/\b} \left(\frac{2e|\p_j(\tfrac{1}{2} + i t + i h_a)|}{r+1}\right)^{2r}.
\]
The claim now follows  by summing over all four cases where $|\p_j(\tfrac{1}{2} + i t + i h_1)|$ and $|\p_j(\tfrac{1}{2} + i t + i h_2)|$ are either smaller or larger than $K_j/10$.
\end{proof}

\begin{lemma}\label{lem:LBInterpErrorControl} For $a = 1,2$
\[
\frac{1}{T} \int_T^{2T} Q_{a,j}(t) dt \ll e^{-K_j},  \quad  \frac{1}{T} \int_T^{2T} Q_{a,j}(t)^2 dt \ll e^{-K_j}.
\]
\end{lemma}
\begin{proof}
To handle the first estimate, note for $0\leq r \leq K_j/\b$ 
\[
\p_j(\tfrac{1}{2} + i t + i h_a)^{K_j + r} = \sum_{\substack{\Omega(n) = K_j + r \\  p|n \Rightarrow p\in (T_{j-1},T_j]}} \frac{(K_j + r)! g(n)}{n^{1/2 + i t + i h_a}}.
\]
This is a Dirichlet polynomial of length at most $T_j^{K_j(1+1/\b)}$  whose coefficients have magnitude $(K_j + r)! g(n)/n^{1/2}$, so the claim follows by the same argument used in lemma 2 of \cite{HS}. 
The only difference is that here $K_j = 250 P_j$.

To handle the second estimate an application of the Cauchy Schwarz inequality yields
\[
Q_{a,j}(t) ^2 \leq \left(\frac{12 |P_j(\tfrac{1}{2} + i t + i h_a)|}{K_j} \right)^{4K_j} \cdot \frac{K_j}{\b} \sum_{r = 0}^{K_j/\b} \left(\frac{2 e |P_j(\tfrac{1}{2} + i t + i h_a)|}{r + 1}\right)^{4r}.
\]
Next note that
\[
\p_j(\tfrac{1}{2} + i t + i h_a)^{2K_j + 2r} = \sum_{\substack{\Omega(n) = 2K_j + 2r \\  p|n \Rightarrow p\in (T_{j-1},T_j]}} \frac{(2K_j + 2r)! g(n)}{n^{1/2 + i t + i h_a}},
\]
which is also a short Dirichlet polynomial. Now the mean value theorem for Dirichlet polynomials gives
\[
\frac{1}{T}\int_T^{2T} |\p_j(\tfrac{1}{2} + i t + i h_a)|^{4K_j + 4r} dt \ll \sum_{\substack{\Omega(n) = 2K_j + 2r \\  p|n \Rightarrow p\in (T_{j-1},T_j]}} \frac{(2K_j + 2r)!^2 g(n)^2}{n} \ll (2K_j + 2r)! P_j^{2K_j + 2r},
\]
where we have used that $g(n)\leq 1$ and the definition of $P_j$.
Whence
\[
\frac{1}{T} \int_T^{2T} Q_{a,j}(t)^2 dt \ll \frac{K_j}{\b} \left(\frac{12}{K_j}\right)^{4K_j} \sum_{r = 0}^{K_j/\b} \left(\frac{2e}{r+1}\right)^{4r} (2K_j + 2r)! P_j^{2K_j + 2r}.
\]
The summand is maximized near $r$ satisfying $r^2 = 4 P_j(2K_j + 2 r).$
Since $K_j = 250P_j$ any  such $r$ necessarily lies in $[2\sqrt{P_j K_j},2.1 \sqrt{P_j K_j}]$, and we conclude in the same manner as lemma 2 of \cite{HS}.

\end{proof}

To conclude, we will use the following splitting lemma, which appears in equation (16) of \cite{HS}
\begin{lemma}\label{lem:Splitting}
Suppose for $1\leq j \leq \ell$ we have $j$ disjoint intervals $I_j$ and Dirichlet polynomials $A_j(s)= \sum_{n} a_j(n)n^{-s}$ such that $a_j(n)$ vanishes unless $n$ is composed of primes in $I_j$.
Then if $\prod_{j\leq \ell} A_j(s)$ is a Dirichlet polynomial of length $\leq N$, then
\begin{align*}
\frac{1}{T}\int_T^{2T} \prod_{j\leq \ell} |A_j(\tfrac{1}{2} + i t)|^2 dt= (1 + O(N/T))\prod_{j\leq \ell}\left(\frac{1}{T}\int_T^{2T}|A_j(\tfrac{1}{2} + i t)|^2 dt\right)
\end{align*}
\end{lemma}
\noindent
In \cite{HS} equation (16), there is an additional factor of $\log N$ in the error term which arises because \cite{HS}  uses a simpler treatment of the mean value theorem for Dirichlet polynomials. 
This however can be easily removed by using a version of the mean value theorem for Dirichlet polynomials with a slightly stronger error term, i.e. theorem 9.1 of \cite{IK}.
In fact a version of this splitting lemma has also appeared in \cite[lemma 14]{FHKUpper}.

To conclude the proof, note the mean value theorem for Dirichlet polynomials gives
\begin{align}
\frac{1}{T}\int_{T}^{2T}& |\n_j(\tfrac{1}{2} + i t + i h_1; \b) \n_j(\tfrac{1}{2}+ i t + i h_2; \b)|^2 dt \nonumber \\
&=(1 + O(T^{-8/9})) \sum_{\substack{p|n \Rightarrow p\in (T_{j-1},T_j] \\ \Omega(n) \leq K_j}} \frac{|b_{h_1,h_2,j}(n)|^2}{n} \nonumber \\
&\leq (1 + O(T^{-8/9}))  \prod_{p\in (T_{j-1},T_j]}\left(1 + \frac{2\b^2(1+\cos((h_1-h_2)\log p))}{p} + O\left(\frac{1}{p^2}\right)\right).
\end{align}
Similar reasoning gives that for $a = 1, 2$
\begin{align}
\frac{1}{T}\int_{T}^{2T} |\n_j(\tfrac{1}{2} + i t + i h_a; \b) |^2 dt &\leq  (1 + O(T^{-8/9}))  \prod_{p\in (T_{j-1},T_j]}\left(1 + \frac{\b^2}{p} + O\left(\frac{1}{p^2}\right)\right) \nonumber \\
\leq &(1 + O(T^{-8/9})) \left(\frac{\log T_j}{\log T_{j-1}}\right)^{\b^2}.
\end{align}
Combining these calculations with lemma \ref{lem:LBInterp}, \ref{lem:LBInterpErrorControl}, and \ref{lem:Splitting} and bounding integrals of products using the Cauchy Schwarz inequality, we find
\begin{align*}
\frac{1}{T} \int_\R M_1 M_2(t)^2 w(t/T) dt & \\
\ll_\b
\prod_{1\leq j \leq \ell} \Bigg( \prod_{T_{j-1} < p \leq T_j}  \Bigg(1 + \frac{2\b^2(1+\cos((h_1-h_2)\log p))}{p} &+ O\left(\frac{1}{p^2}\right)\Bigg)  + O(e^{-K_j / 10}) \Bigg) \\
\ll_\b  (\log T)^{2\b^2} |\zeta(1+1/\log T_\ell+i(h_1&-h_2))|^{2\b^2}.
\end{align*}
This completes the proof of Proposition \ref{prop:lowerDPMVB<1}.

\subsection{Proof of Theorem \ref{thm:shiftedMoments} for $\b \geq 1$}

The case of $\b \geq 1$  a bit simpler than the case $\b \leq 1$.
Now let $T_0 = \b^4 e^2$, let $\ell$ be the largest integer such that $\log_\ell T \geq 10^4,$ for $1\leq j \leq \ell$ set
\[
T_j = \exp\left(\frac{\log T}{\b^2(\log_j T)^2}\right),
\]
and take $K_j = 250 \b^2 P_j$.
With this choice of parameters, $\n(s;\a)$ is a Dirichlet polynomial of length at most $ T^{1/18}$.
Now because $\b \geq 1$, Hölder's inequality gives

\begin{align} \label{eqn:InterpB>1} 
\Bigg|\int_\R \zeta(\tfrac{1}{2} + it + i h_1) \zeta(\tfrac{1}{2} + it &+ i h_2) \n(\tfrac{1}{2}+ i t + i h_1;\b-1) \n(\tfrac{1}{2}+ i t + i h_2;\b-1) \\
&\quad\times \n(\tfrac{1}{2} - i t - i h_1;\b) \n(\tfrac{1}{2}- i t - i h_2;\b) w(t/T)~ dt  \Bigg| \nonumber \\
\leq \bigg(\int_\R &|\zeta(\tfrac{1}{2} + it + i h_1)  \zeta(\tfrac{1}{2} + it + i h_2)|^{2\b} w(t/T) ~dt\bigg)^{\tfrac{1}{2\b}} \nonumber  \\
\times\bigg(\int_\R |\n(\tfrac{1}{2}+ i t + i h_1;&\b-1)\n(\tfrac{1}{2} + i t + i h_2;\b-1)  \nonumber\\
\times \n(\tfrac{1}{2}&+ i t + i h_1;\b)\n(\tfrac{1}{2}+ i t + i h_2;\b)|^{\frac{2\b}{2\b - 1}} w(t/T) ~ dt \bigg)^{1-\tfrac{1}{2\b}}.\nonumber
\end{align}

Therefore, if we denote
\[
N_{a,j}(t) := |\n_j(\tfrac{1}{2}+ i t + i h_1;\b-1) \n_j(\tfrac{1}{2}+ i t + i h_1;\b)|^{\tfrac{2\b}{2\b - 1}}
\]
all we must show is 

\begin{prop}\label{prop:DPMVB>1}
For $\b \geq 1$ and $|h_1|, |h_2| \leq (\log T)^\t$ 
\begin{align*}
\frac{1}{T} \int_\R  \prod_{j\leq \ell}  &N_{1,j}(t)  N_{2,j}(t)  w(t/T) ~ dt \\
&\ll_\b  (\log T)^{2\b^2} |\zeta(1+1/\log T +i(h_1-h_2))|^{2\b^2}.
\end{align*}
\end{prop}
\noindent
This is a lot like Proposition \ref{prop:lowerDPMVB<1}, but the proof is even simpler.
The analog of Lemma \ref{lem:LBInterp} is

\begin{lemma}\label{lem:b>1LBInterp}
For $j\leq \ell$
\begin{align*}
N_1,j(t) N_{2,j}(t) \leq  
|\n_j(\tfrac{1}{2} + i t &+ i h_1; \b) \n_j(\tfrac{1}{2}+ i t + i h_2; \b)|^2 (1+O(e^{-K_j/10})) \\
 + O\big(R_{1,j}(t) |\n_j(\tfrac{1}{2}&+ i t + i h_2; \b)|^2(1+O(e^{-K_j/10})) \\
+ R_{2,j}(t) |\n_j(\tfrac{1}{2}+ &i t + i h_2; \b)|^2(1+O(e^{-K_j/10}))\\
 +  &R_{1,j}(t) R_{2,j}(t) \big),
\end{align*} 
where for $a = 1,2$
\[
R_{a,j}(t) := \left(\frac{12 \b|\p_j(\tfrac{1}{2} + i t + i h_a)|}{K_j}\right)^{2 K_j} .
\]
\end{lemma}
\begin{proof}
Following the proof of Lemma \ref{lem:LBInterp}, if $|\p_j(\tfrac{1}{2} + i t + i h_a)|\leq K_j/10$, then we may write 
\begin{align*}
N_{a,j}(t) = |\n_j(\tfrac{1}{2} + i t + i h_a; \b)|^2 (1+O(e^{-K_j/10})).
\end{align*} 
In the case $|\p_j(\tfrac{1}{2} + i t + i h_a)| > K_j/10$ for some $a$ then simply note
for $\a = \b-1$ or $\a= \b$
\begin{align*}
|\n_j(\tfrac{1}{2} &+ i t + i h_a;\a)|^{\tfrac{2\b}{2\b-1}} \leq \left( \sum_{r = 0}^{K_j} \frac{\b^r |\p_j(\tfrac{1}{2}+it + i h_a)|^r}{r!}\right)^{\tfrac{2\b}{2\b-1}} \\
&\leq  \left( \sum_{r = 0}^{K_j} \frac{\b^r |\p_j(\tfrac{1}{2}+it + i h_a)|^r}{r!}\right)^{2} \leq \left(\frac{12\b|\p_j(\tfrac{1}{2}+it + i h_a)|}{K_j}\right)^{2K_j},
\end{align*}
where we have used that $\b \geq 1$ and that $|\p_j(\tfrac{1}{2}+it + i h_a)|$ is large.
The claim now follows.
\end{proof}

\noindent
The analog of Lemma \ref{lem:LBInterp} is

\begin{lemma}\label{lem:b>1LBInterpErrorControl} For $a = 1,2$
\[
\frac{1}{T} \int_T^{2T} R_{a,j}(t) dt \ll e^{-K_j},  \quad  \frac{1}{T} \int_T^{2T} R_{a,j}(t)^2 dt \ll e^{-K_j}.
\]
\end{lemma}
\noindent
The proof is a simpler version of  Lemma \ref{lem:LBInterp}, and is omitted.
Now we may conclude the proof of Proposition \ref{prop:DPMVB>1} when $\b \geq 1$ by combining Lemma \ref{lem:b>1LBInterp}, Lemma \ref{lem:b>1LBInterpErrorControl} and Lemma \ref{lem:Splitting} with the mean value theorem for Dirichlet polynomials as we did in the $\b \leq 1$ case.
This concludes the proof of \ref{prop:DPMVB>1}, hence the proof of Theorems \ref{thm:mainLowerBound}  and \ref{thm:shiftedMoments}.


\begin{thebibliography}{99}
\bibitem{FHKLeadingOrder}
L.-P. Arguin, D. Belius, P. Bourgade, M. Radziwiłł, and K. Soundararajan. \textit{Maximum of the Riemann zeta function on a short interval of the critical line}. Comm. Pure Appl. Math., (3) 72 (2019): 500– 535.


\bibitem{FHKUpper}
L.-P. Arguin, P. Bourgade, and M. Radziwiłł. \textit{The Fyodorov-Hiary-Keating Conjecture. I.} Preprint arXiv:2007.00988 (2020).

\bibitem{AOR}
L-P. Arguin, F. Ouimet, and M. Radziwiłł. \textit{Moments of the Riemann zeta function on short intervals of the critical line.} arxiv:1901.04061 (2019).



\bibitem{AKUnitary}
Assiotis, T., Keating, J.P. \textit{Moments of moments of characteristic polynomials of random unitary matrices and lattice point counts.} Random Matrices Theory  Appl. (2020): 2150019.

\bibitem{BKUnitary}
Bailey, E.C., Keating, J.P. \textit{ On the moments of the moments of the characteristic polynomials of random unitary matrices.} Comm. Math. Phys. (2) 371 (2019): 689–726 .

\bibitem{BKZetaMoM}
Bailey, E.C., Keating, J.P. \textit{On the moments of the moments of $\zeta(1/2 + it)$. } arXiv preprint arXiv:2006.04503 (2020).


\bibitem{BCHBTwisted2nd}
R. Balasubramanian, J.B. Conrey, D. R. Heath-Brown, \textit{Asymptotic mean square of the product of the Riemann zeta-function and a Dirichlet polynomial}, J. Reine Angew. Math.  {357} (1985): 161–181.

\bibitem{BettinUS} S. Bettin, \textit{The second moment of the Riemann zeta function with unbounded shifts.} Int. J. Number Theory 6,  no. 8 (2010): 1933–1944. 

\bibitem{BBLRTwisted4}
S. Bettin, H. M. Bui, X. Li, M. Radziwiłł,\textit{ A quadratic divisor problem and moments of the Riemann zeta-function}, J. Eur. Math. Soc. Electronically published on August 10, 2020. doi: 10.4171/JEMS/999 (to appear in print).

\bibitem{BCRTwisted2}
S. Bettin, V. Chandee, M. Radziwiłł, \textit{The mean square of the product of $\zeta(s)$ with Dirichlet polynomials}, J. Reine Angew. Math,  {729}  (2017): 51–79.

\bibitem{Bourgade}
P. Bourgade,\textit{ Mesoscopic fluctuations of the zeta zeros,} Probab. Theory Related Fields {148},
no. 3-4  (2010): 479–500.

\bibitem{Chandee}
V. Chandee. \textit{On the correlation of shifted values of the Riemann zeta function}. Q. J. Math. 62, no. 3 (2011): 545-572.

\bibitem{CK}
 Claeys, T., Krasovsky, I. \textit{Toeplitz determinants with merging singularities.} Duke Math. J. (15) 164 (2015): 2897–2987.
 
\bibitem{Conrey2/5}
J. B. Conrey, \textit{ More than two fifths of the zeros of the Riemann zeta function are on the critical line}, J.
Reine Angew. Math.  {339} (1989): 1–26.


\bibitem{Conrey4thMoments}
J. B. Conrey, \textit{The fourth moment of derivatives of the Riemann zeta-function}, Q. J. Math.  {39}, no. 1 (1988): 21-36.

\bibitem{CFKRSRMT}
J. B. Conrey, D. W. Farmer, J. P.  Keating, M. O. Rubinstein, and N. C. Snaith, \textit{Autocorrelation of random matrix polynomials}, Commun. Math. Phys  {237} 3 (2003): 365-395.

\bibitem{CFKRS}
J. B. Conrey,  D. W. Farmer,  J. P. Keating, M. O. Rubinstein, and  N. C. Snaith,\textit{ Integral moments of $L$-functions}, Proc. Lond. Math. Soc.  {91} no. 3 (2005): 33–104.

\bibitem{CGJM}
J. B. Conrey, A. Ghosh, \textit{On the mean values of the zeta-function II}, Acta Arith.  {52}, no. 4 (1989): 367-371.

\bibitem{Curran}
M. J. Curran, \textit{Upper bounds for fractional joint moments of the Riemann zeta function}, Acta Arith. {204} (2022): 83-96.

\bibitem{Fahs}
Fahs, B. \textit{Uniform asymptotics of Toeplitz determinants with Fisher–Hartwig singularities.} arXiv preprint arXiv:1909.07362 (2019).

\bibitem{FHK}
Y. V. Fyodorov, J. P. Keating,\textit{ Freezing transitions and extreme values: random matrix theory, and disordered landscapes,}
Philos. Trans. Roy. Soc. A: 372 (2014): 
20120503.


\bibitem{FK}
Y. V. Fyodorov, G. A. Hiary, J. P. Keating, \textit{Freezing transition, characteristic polynomials of random matrices, and the Riemann zeta function,} Phys. Rev. Lett. 108 , no. 17 (2012): 170601.


\bibitem{GS}
A. Granville and K. Soundararajan, Multiplicative number theory: The pretentious approach. To appear.


\bibitem{InghamMV}
A. E. Ingham, \textit{Mean-value theorems in the theory of the Riemann zeta function}, Proc. Lond. Math. Soc.  {27} (1926): 273–300.

\bibitem{IK}
H. Iwaniec and E. Kowalski. \textit{Analytic number theory},   American Mathematical Society, vol. 53, Providence, RI, 2004.

\bibitem{HL}
G. H. Hardy and J. E. Littlewood, \textit{Contributions to the theory of the Riemann zeta-function and the theory of the distribution of primes.} Acta Math. 41 (1918): 119 - 196.

\bibitem{HarperPartition}
A. J. Harper. \textit{On the partition function of the Riemann zeta function, and the Fyodorov–Hiary– Keating conjecture.} arXiv preprint arXiv:1906.05783 (2019).

\bibitem{Harper}
A. J. Harper, \textit{Sharp conditional bounds for moments of the Riemann zeta function}, arXiv.1305.4618 (2013).

\bibitem{HarperShortIntervals}
A. J. Harper. \textit{The Riemann zeta function in short intervals [after Najnudel, and Arguin, Belius, Bourgade, Radziwiłł, and Soundararajan].} arxiv:1904.08204, Séminaire Bourbaki, to appear in Astérisque, 2019.

\bibitem{HRS}
W. Heap., M. Radziwiłł, K. Soundararajan, \textit{Sharp upper bounds for fractional moments of the Riemann zeta function}, Q. J. Math.  {70}, no. 4 (2019): 1387–1396.

\bibitem{HS}
W. Heap., K. Soundararajan, \textit{Lower bounds for moments of zeta and L-functions revisited}, Mathematika {1}, no. 68 (2022): 1-14.


\bibitem{HYTwisted4}
C. P. Hughes, M. P. Young, \textit{The twisted fourth moment of the Riemann zeta function}, J. Reine
Angew. Math.  {641} (2010): 203–236.

\bibitem{KS}
J. P. Keating, N. C. Snaith, \textit{Random matrix theory and $\zeta(\tfrac{1}{2} + it)$}, Comm. Math. Phys.  {214} (2000): 57–89.

\bibitem{KSLFns}
J.P. Keating, N.C. Snaith, \textit{Random matrix theory and $L$-Functions at $s = \tfrac{1}{2}$}, Comm. Math. Phys.  {214}, no. 1 (2000): 91-100.

\bibitem{KW}
J. P. Keating, M. D. Wong. \textit{On the Critical–Subcritical Moments of Moments of Random Characteristic Polynomials: A GMC Perspective}. Comm. Math. Phys. 394, no. 3 (2022): 1247-1301.

\bibitem{Koukoulopoulos}
D. Koukoulopoulos, \textit{ Pretentious multiplicative functions and the prime number theorem for arithmetic
progressions}, Compos. Math. 149, no. 7 (2013): 1129–1149.

\bibitem{Kovaleva}
V. Kovaleva. \textit{Topics in probabalistic number theory}, PhD thesis, University of Oxford (2022).


\bibitem{Levinson}
N. Levinson, \textit{More than one third of the zeros of Riemann’s zeta function are on $\sigma = \tfrac{1}{2}$}, Adv. Math.  {13} (1974): 383–436.

\bibitem{Najnudel}
J. Najnudel.\textit{ On the extreme values of the Riemann zeta function on random intervals of the critical line}. Probab. Theory Related Fields, 172(1-2) (2018): 387–452.

\bibitem{NSW}
N. Ng, Q. Shen, P.-J. Wong. \textit{Shifted moments of the Riemann zeta function.} arXiv preprint arXiv:2206.03350 (2022).

\bibitem{PZCBE}
E. Paquette, O. Zeitouni.  \textit{The extremal landscape for the C$\b$E ensemble}. arXiv preprint arXiv:2209.06743 (2022).

\bibitem{PZCUE}
E. Paquette, O. Zeitouni. \textit{The maximum of the CUE field}. Int. Math. Res. Not. IMRN, 16 (2018): 5028-5119.


\bibitem{LDSCLT}
M. Radziwiłł, \textit{Large deviations in Selberg’s central limit theorem}, arXiv.1108.5092 (2011).

\bibitem{RS-LB}
M. Radziwiłł, K. Soundararajan, \textit{Continuous lower bounds for moments of zeta and L-functions}, Mathematika, 59 no. 1 (2013): 119–128.

\bibitem{RS-EC}
M. Radziwiłł, K. Soundararajan, \textit{Moments and distribution of central $L$-values of quadratic
twists of elliptic curves}, Invent. Math.  {202} no. 3 (2015): 1029–1068.


\bibitem{SRHMoments}
K. Soundararajan, \textit{Moments of the Riemann zeta function}, Ann. of Math.  {170} no. 2 (2009): 981–993.

\bibitem{Titchmarsh}
E. C. Titchmarsh, \textit{The Theory of the Riemann Zeta-Function}, 2nd ed., Oxford University Press, New York, (1986). 

\bibitem{YoungSL}
M. P. Young, \textit{A short proof of Levinson’s theorem}, Arch. Math.  {95} (2010): 539–548.
\end{thebibliography}
\end{document}